\documentclass[11pt]{amsart}
\usepackage{bm}
\usepackage{chngpage}
\usepackage{graphicx}
\usepackage{eucal}
\usepackage{amssymb}
\usepackage{color,xcolor}
\usepackage{graphicx,epstopdf}
\usepackage{epsfig}
\usepackage{subfig}
\usepackage{caption}
\usepackage{mathrsfs}
\usepackage{tikz}
\usepackage{hyperref}
\usepackage{tikz-cd}
\usepackage{comment}
\usepackage[misc,geometry]{ifsym}

\newcommand\ran{\mathcal{R}}
\renewcommand\ker{\mathcal{N}}

\allowdisplaybreaks

\newtheorem{theorem}{Theorem}[section]

\newtheorem{scheme}{Scheme}
\theoremstyle{definition}

\theoremstyle{remark}
\newtheorem{remark}[theorem]{Remark}
\newtheorem{lemma}{Lemma}[section]

\numberwithin{equation}{section}%
\numberwithin{table}{section}%
\numberwithin{figure}{section}

\def\3bar{{|\hspace{-.02in}|\hspace{-.02in}|}}

\def\d{\text{d}}

\def\d{\mathrm{d}}

\def\d{\mathrm{d}}

\def\curl{\text{curl}}

\def\div{\operatorname{div}}
\def\curl{\operatorname{curl}}

\def\3bar{{|\hspace{-.02in}|\hspace{-.02in}|}}

\newcommand\grad{\operatorname{grad}}
\renewcommand\div{\operatorname{div}}

\def\d{\text{d}}

\allowdisplaybreaks

\def\3bar{{|\hspace{-.02in}|\hspace{-.02in}|}}

\def\d{\text{d}}

\newcommand\rot{\operatorname{rot}}
\renewcommand\div{\operatorname{div}}

\begin{document}
\title[Spurious solutions for high order curl problems]{Spurious solutions for high order curl problems}
	

\author{Kaibo Hu}
\email{kaibo.hu@maths.ox.ac.uk}
\address{Mathematical Institute, University of Oxford, Oxford
OX2 6GG, UK}
\author{Qian Zhang (\Letter)}
\email{qzhang15@mtu.edu}
\address{Department of Mathematical Sciences, Michigan Technological University, Houghton, MI 49931, USA. }
\author{Jiayu Han}
\email{hanjiayu@csrc.ac.cn}
\address{Beijing Computational Science Research Center, Beijing, 100193, China;
School of Mathematical Sciences, Guizhou Normal University, 550001, China.
}
\author{Lixiu Wang}
\email{lxwang@csrc.ac.cn}
\address{School of Mathematics and Physics, University of Science and Technology Beijing, Beijing 100083, China;
Beijing Computational Science Research Center, Beijing, 100193, China.
}

\author{Zhimin Zhang}
\email{zmzhang@csrc.ac.cn; ag7761@wayne.edu}

\address{Beijing Computational Science Research Center, Beijing, China; Department of Mathematics, Wayne State University, Detroit, MI 48202, USA}
\thanks{This work is supported in part by the National Natural Science Foundation of China grants NSFC 11871092 (Zhimin Zhang), NSAF U1930402 (Zhimin Zhang), NSFC 12001130 (Jiayu Han), and NSFC 12101036 (Lixiu Wang).}


\maketitle
\begin{abstract}
We investigate numerical solutions of high order $\curl$ problems with various  formulations and finite elements. We show that several classical conforming finite elements lead to spurious solutions, while mixed  formulations with finite elements in complexes solve the problems correctly. 
 To explain the numerical results, we clarify the cohomological structures in high order $\curl$ problems by relating the partial differential equations to the Hodge-Laplacian boundary problems of the $\grad\curl$-complexes.
\end{abstract}


\section{Introduction}

Spurious solutions usually refer to numerical solutions that converge to a ``wrong solution'' of partial differential equations (PDEs) or variational problems. Such solutions are dangerous for applications because of the lack of a visible condition (e.g., non-convergence or instability) to tell them from the correct solutions. Spurious solutions of eigenvalue problems are called ``the plague'' \cite{bossavit1990solving} (or, sometimes called the ``vector parasites'' \cite{sun1995spurious}) since the early days of computational electromagnetism. Several possible reasons were conjectured, for example, violation of the divergence-free condition of the discrete magnetic field (see, e.g., \cite{bossavit1990solving,sun1995spurious,winkler1984elimination} and the references therein). In \cite{sun1995spurious}, Sun et al. pointed out that the cause of the problem is not that the precise solenoidal condition fails at the discrete level, but rather a poor discretization of the kernel and range of the operators. In a modern language, this means that the Lagrange finite element fails to fit in a de~Rham complex with local finite element spaces. In \cite{sun1995spurious} the authors also discussed a smoother de~Rham complex (Stokes complex) on the Powell-Sabin split \cite{guzman2020exact} and showed that the Lagrange vector element  avoids spurious solutions on this mesh. This demonstrates that the $C^{0}$ continuity of the Lagrange elements is not an obstacle for eigenvalue problems as long as the underlying complexes exist and appropriate schemes are used (see also \cite{duan2019new,boffi2020convergence}). 

Spurious solutions also appear in source problems. Solving the vector Laplacian problem using the primal formulation and the Lagrange finite element leads to wrong solutions on an annulus or an L-shape domain \cite{arnold2018finite,arnold2010finite}. On an annulus domain,  the Lagrange elements cannot approximate the harmonic forms due to the lack of a cohomological structure at the discrete level. For the $L$-shape domain, the problem  lies in an approximation issue. Since $[H^{1}]^{3}\cap H_{0}(\div)$ is a closed subspace in $H(\curl)\cap H_{0}(\div)$ which in general is not dense in the norm defined by the primal formulation, numerical solutions of the primal formulation with $C^{0}$ elements cannot approximate true solutions with singularity. 

For the Maxwell equations, the N\'ed\'elec element fixes both issues: on the one hand,  it fits into a de~Rham sequence with finite element spaces; on the other hand, generally the N\'ed\'elec element does not belong to $H^{1}$ \cite{boffi2010finite}. So the finite element scheme converges correctly to singular solutions. The success of the N\'ed\'elec element inspired the development of discrete differential forms \cite{bossavit1988whitney,hiptmair2002finite} and the numerical tests on spurious solutions demonstrate the power of the finite element exterior calculus (FEEC) \cite{arnold2018finite,arnold2006finite,arnold2010finite}.

Recently, high order problems involving the fourth order curl operator  $(\curl)^4$, or its variant,  $- \!\curl \!\Delta\!\curl$, draw attention. These operators and PDEs appear in some magnetohydrodynamics and continuum models \cite{Cakoni2017A,chacon2007steady,mindlin1962effects,Monk2012Finite,park2008variational}.  It is speculated that spurious solutions may appear in certain discretizations. For example, Zhang and Zhang \cite{zhang2020order} anticipated that spurious solutions may appear in the 3D quad-curl source problem due to an approximation issue. Delicate finite elements and mixed schemes were constructed for solving high order curl problems (see, e.g., \cite{Brenner2017Hodge,cao2021error,huang2020nonconforming,sun2016mixed,zhang2018mixed,zheng2011nonconforming}). Nevertheless, to the best of our knowledge, there are no numerical examples and detailed analysis to support the speculation of spurious solutions. For vector Laplacian problems, after showing that the N\'ed\'elec element avoids spurious solutions, Bossavit \cite{bossavit1990solving} asked the question: can we rigorously prove that spurious modes necessarily appear in certain numerical methods.  To the best of our knowledge, this question is still open for both the vector Laplacian and higher order problems. 

In this paper, we provide numerical evidence and analysis to show that spurious solutions do exist for high order curl eigenvalue and source problems. This is for several reasons, for example, the lack of the underlying cohomological structures and problems caused by low regularity. We will show that the high order curl problems are naturally related to the $\grad\curl$-complex \cite[(46)]{arnold2020complexes}, and FEEC formulations fix the spurious modes. 
Specifically, we compare the primal and/or mixed formulations with the $H(\grad\curl)$-conforming element \cite{hu2020simple}, the $H^{1}(\curl)$-conforming element \cite{falk2013stokes}, and the $H^{2}$-conforming (Argyris) element. As we shall see, mixed formulations with finite elements that fit into complexes (c.f.,\cite{hu2020family,hu2020simple,neilan2015discrete,zhang2020curl}) lead to correctly convergent solutions, while other combinations may produce spurious solutions. 

We also remark that for the biharmonic equation 
\begin{equation}\label{biharmonic}
\Delta^{2}u=f,
\end{equation}
 another type of spurious solutions may appear \cite{gerasimov2012corners,zhang2008invalidity} if one decomposes the problem and seeks $u, \sigma \in H^{1}$ which solves 
 \begin{equation}\label{biharmonic-mixed}
 \Delta u =\sigma, \quad \Delta \sigma=f.
 \end{equation} 
 This is because the solution for \eqref{biharmonic} from a primal weak form is in $H^{2}$, while on domains with corners the solution $u$ from \eqref{biharmonic-mixed} may not be in $H^{2}$. Thus mixed formulations based on \eqref{biharmonic-mixed} are not equivalent to \eqref{biharmonic} in general. Nevertheless, this is a difference at the continuous level and we do not pursue methods that decouple the high order curl operators, so we will not discuss this issue in this paper.

The rest of the paper will be organized as follows. In Section 2, we show numerical examples of the spurious solutions by comparing various schemes and finite elements. In Section 3, we relate the high order $\curl$ problem to the $\grad\rot$-complex and the Hodge-Laplacian boundary value problem. In Section 4, we show the convergence of the mixed formulations with finite element complexes \cite{hu2020simple} and analyze the spurious solutions. We provide concluding remarks in Section 5.


\section{Finite element discretization for high order curl problems}\label{experiments}

Unless otherwise specified, in this paper $\Omega$ is a bounded Lipschitz domain in two space dimensions (2D) with unit outward normal vector $\bm n$ and unit tangential vector $\bm \tau$  on its boundary $\partial\Omega$.
We adopt standard notation for Sobolev spaces such as $H^s(\Omega)$ with norm $\|\cdot\|_{s}$ and inner product $(\cdot,\cdot)_{s}$. When $s=0$, $H^0(\Omega)$ coincides with $L^2(\Omega)$, in which case we omit the subscript $0$ in the notation of the norm and the inner product. We use $H_0^s(\Omega)$ to denote the spaces with vanishing trace.

Define 
\begin{align*}
H_{\rot}(\grad\rot;\Omega)&:=\{\bm u\in H(\grad\rot;\Omega) : \rot\bm u=0 \mbox{ on } \partial \Omega\},\\
X_{\rot}&:=H_{\rot}(\grad\rot;\Omega)\cap H_0(\div;\Omega).
\end{align*}
We consider the following source problem on a simply-connected domain. Denote by $\mathbb V$ the space of vectors. 
For $\bm f\in L^2(\Omega)\otimes \mathbb V$ (tensor product of $L^2(\Omega)$ and $\mathbb V$), we seek $\bm u\in X_{\rot}$  such that 
\begin{equation}\label{prob-source1}
\begin{split}
-\curl \Delta \rot\bm u-\grad\div\bm u &=\bm f\ \text{in}\;\Omega,
\end{split}
\end{equation}
with the boundary conditions 
\begin{align}\label{bcs}
\Delta\rot\bm u=0,\ 
\rot\bm u=0, \text{ and }
\bm u\cdot\bm n=0\ \text{on}\;\partial \Omega.
\end{align}
The  primal variational formulation is to
seek $\bm u\in X_{\rot}$  such that
\begin{equation}\label{pvprob1}
\begin{split}
(\grad\rot\bm u,\grad\rot\bm v)+(\div\bm u,\div\bm v)&=(\bm f,\bm v), \quad \forall \bm v\in X_{\rot}.
\end{split}
\end{equation}
Let $\sigma=-\div\bm u$. Then the mixed variational formulation
seeks $(\bm u,\sigma)\in  H_{\rot}(\grad\rot;\Omega)\times H^1(\Omega)$ such that
\begin{equation}\label{mvprob1}
\begin{split}
(\grad\rot\bm u,\grad\rot\bm v)+(\grad\sigma,\bm v)&= (\bm f,\bm v), \quad \forall \bm v\in  H_{\rot}(\grad\rot;\Omega),\\
(\bm u,\grad \tau)-(\sigma,\tau)&=0,\quad \forall \tau\in H^1(\Omega).
\end{split}
\end{equation}

We also consider the corresponding eigenvalue problem on a general domain. The strong formulation is to seek $(\lambda,\bm u)\in \mathbb{R}\times X_{\rot}$  such that 
\begin{equation}\label{prob-eig}
-\curl \Delta \rot\bm u-\grad\div\bm u =\lambda\bm u\quad \text{in}\;\Omega,
\end{equation}
with the boundary conditions \eqref{bcs}.
The  primal variational formulation is to
seek $(\lambda;\bm u)\in \mathbb R\times X_{\rot}$ such that
\begin{equation}\label{pvprob-eig}
\begin{split}
(\grad\rot\bm u,\grad\rot\bm v)+(\div\bm u,\div\bm v)&=\lambda(\bm u,\bm v), \ \forall \bm v\in X_{\rot}.
\end{split}
\end{equation}
The mixed variational formulation
seeks $(\lambda,\bm u,\sigma)\in \mathbb R\times H_{\rot}(\grad\rot;\Omega)\times H^1(\Omega)$ s.t.
\begin{equation}\label{mvprob-eig}
\begin{split}
(\grad\rot\bm u,\grad\rot\bm v)+(\grad\sigma,\bm v)&= \lambda(\bm u,\bm v), \ \forall \bm v\in  H_{\rot}(\grad\rot;\Omega),\\
(\bm u,\grad \tau)-(\sigma,\tau)&=0,\ \forall \tau\in H^1(\Omega).
\end{split}
\end{equation}

We consider four finite element methods (FEMs) to solve the problems \eqref{prob-source1} and \eqref{prob-eig}: a primal formulation with the $H^{2}$-conforming Argyris element, a mixed formulation with the $\grad\rot$-conforming element \cite{WZZelement}, and the mixed and primal formulations with the $H^1(\rot)$-conforming element \cite{falk2013stokes}. Let $\mathcal T_h$ be a partition of the domain $\Omega$ consisting of shape regular  triangles, and let $\mathcal V$ be the set of vertices. We denote by $P_k$ the space of polynomials with degree no larger than $k$. For $K\in \mathcal T_h$, define
\begin{align*}
	\mathcal R_k(K)&:=\grad P_k(K)\oplus P_{k-1}(K) \bm x^{\perp}, \ k\geq 4,\\
	\mathcal W_k(K):=\Big\{\bm v\in P_k(K)&\otimes\mathbb R^2:\rot\bm v\in P_{k-3}(K)\cup \left[P_{k-1}(K)\cap H_0^1(K)\right]\Big\}, \ k\geq 4.
\end{align*}
\begin{remark}
	We can use the Poincar\'e operator $\mathfrak p u=\int_0^1u(t\bm x)t\bm x^{\perp}\d t$ to construct $\mathcal W_k(K)$. For example, when $k=4$,
	$\mathcal W_k(K)=\grad P_5(K)\oplus P_{1}(K) \bm x^{\perp}\oplus \{\mathfrak p(\lambda_1\lambda_2\lambda_3)\}$.
\end{remark}
The $H(\grad\rot)$- and $H^{1}(\rot)$- conforming finite element spaces on the partition $\mathcal T_h$ are listed as follows: 
\begin{align*}
V_h&=\big\{\bm{v}_h\in  H_{\rot}(\grad\rot;\Omega):\ \bm v_h|_K\in \mathcal{R}_k(K),\ \forall K\in\mathcal{T}_h\big\},\\
V^1_h=\big\{\bm{v}_h&\in  H^1(\rot; \Omega)\cap H_{\rot}(\grad\rot;\Omega):\ \bm v_h|_K\in \mathcal W_k(K),\ \forall K\in\mathcal{T}_h\big\}.
\end{align*}
\begin{remark}
	We can also choose $V_h$ as the other $\grad\rot$-conforming finite element spaces in \cite{hu2020simple}.
\end{remark}
We also define the following two finite element spaces for the mixed schemes.
\begin{align*}
		 &S_h=\{{w}_h\in  H^1(\Omega):\  w_h|_K\in P_k(K), \ \forall K\in\mathcal{T}_h\},\\
		 S_h^1=\{{w}_h\in  &H^2(\Omega):\  w_h|_K\in P_k(K), \ \forall K\in\mathcal{T}_h,\ w_h\in C^2(\mathcal V)\} \text{ for }k\geq 5.
\end{align*}

The vector-valued $H^2$-conforming finite element space is defined as
\[V_h^{Arg}=S_h^1\otimes \mathbb R^2.\]
Define
 \begin{align*}
 \mathring V^1_h&=\{\bm{v}_h\in V_h^1: \bm v_h\cdot\bm n=0 \text{ on }\partial \Omega\},\\
\mathring V_h^{Arg}&=\{\bm{v}_h\in V_h^{Arg}: \bm v_h\cdot\bm n=0\text{ and } \rot\bm v_h=0\text{ on }\partial \Omega\}.
\end{align*}

\subsection{Source problem}
We are in a position to present the four finite element schemes for the problem \eqref{prob-source1}. 
\begin{scheme}[Mixed formulation with the $H(\grad\rot)$-conforming element]\label{algo1}
Find $(\bm u_h,\sigma_{h}) \in  V_h\times S_{h}$  such that 
\begin{equation*}
\begin{split}
(\grad\rot\bm u_h,\grad\rot\bm v_h)+(\grad\sigma_h,\bm v_h)&= (\bm f,\bm v_h), \ \forall \bm v_h\in V_h,\\
(\bm u_h,\grad \tau_h)-(\sigma_h,\tau_h)&=0,\ \forall \tau_h\in S_h.
\end{split}
\end{equation*}
\end{scheme}
\begin{scheme}[Mixed formulation with the $H^1(\rot)$-conforming element]\label{algo2}
Find $(\bm u_h,\sigma_{h}) \in  V_h^1\times S_{h}^1$ such that 
\begin{equation*}
\begin{split}
(\grad\rot\bm u_h,\grad\rot\bm v_h)+(\grad\sigma_h,\bm v_h)&=(\bm f,\bm v_h), \ \forall \bm v_h\in V_h^1,\\
(\bm u_h,\grad \tau_h)-(\sigma_h,\tau_h)&=0,\ \forall \tau_h\in S_h^1.
\end{split}
\end{equation*}
\end{scheme} 		
\begin{scheme}[Primal formulation with the $H^1(\rot)$-conforming element]\label{algo3}
Find $\bm u_h \in \mathring V_h^{1}$  such that 
\begin{equation*}
\begin{split}
(\grad\rot\bm u_h,\grad\rot\bm v_h)+(\div\bm u_h,\div\bm v_h)&=(\bm f, \bm v_h),\quad \forall \bm v_h\in \mathring V_h^{1}.
\end{split}
\end{equation*}
\end{scheme}
\begin{scheme}[Primal formulation with the $H^2$-conforming (Argyris) element]\label{algo4}
Find $\bm u_h \in \mathring V_h^{Arg}$ such that 
\begin{equation*}
\begin{split}
(\grad\rot\bm u_h,\grad\rot\bm v_h)+(\div\bm u_h,\div\bm v_h)
&=(\bm f, \bm v_h),\quad \forall \bm v_h\in \mathring V_h^{Arg}.\end{split}
\end{equation*}
\end{scheme}
\begin{figure}[!h]
  \subfloat[Scheme \ref{algo1}]{
	\begin{minipage}[c][0.8\width]{
	   0.4\textwidth}
	   \centering
	   \includegraphics[width=1\textwidth]{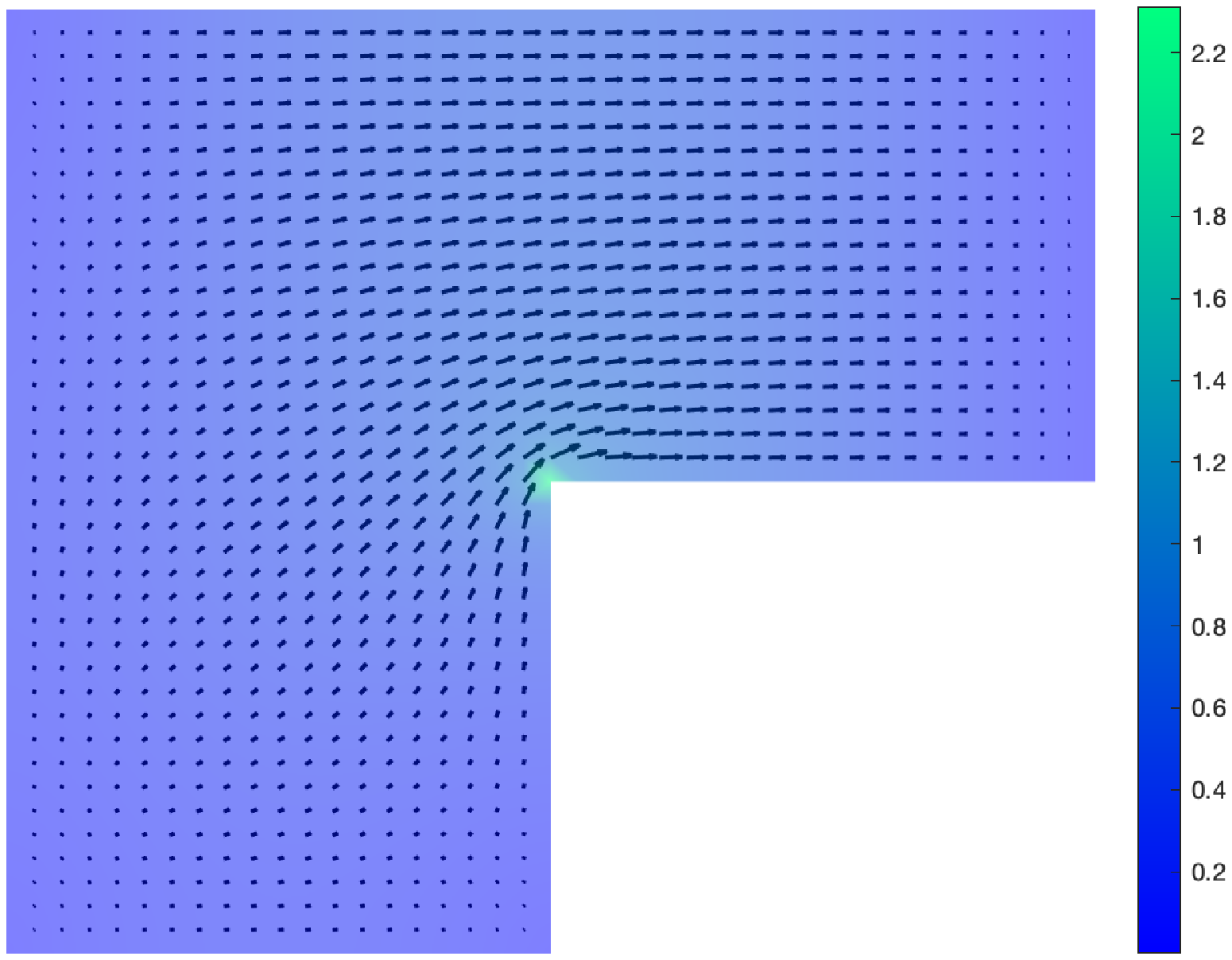}
	\end{minipage}}
 \hfill 	
    \subfloat[Scheme \ref{algo2}]{
	\begin{minipage}[c][0.8\width]{
	   0.4\textwidth}
	   \centering
	   \includegraphics[width=1\textwidth]{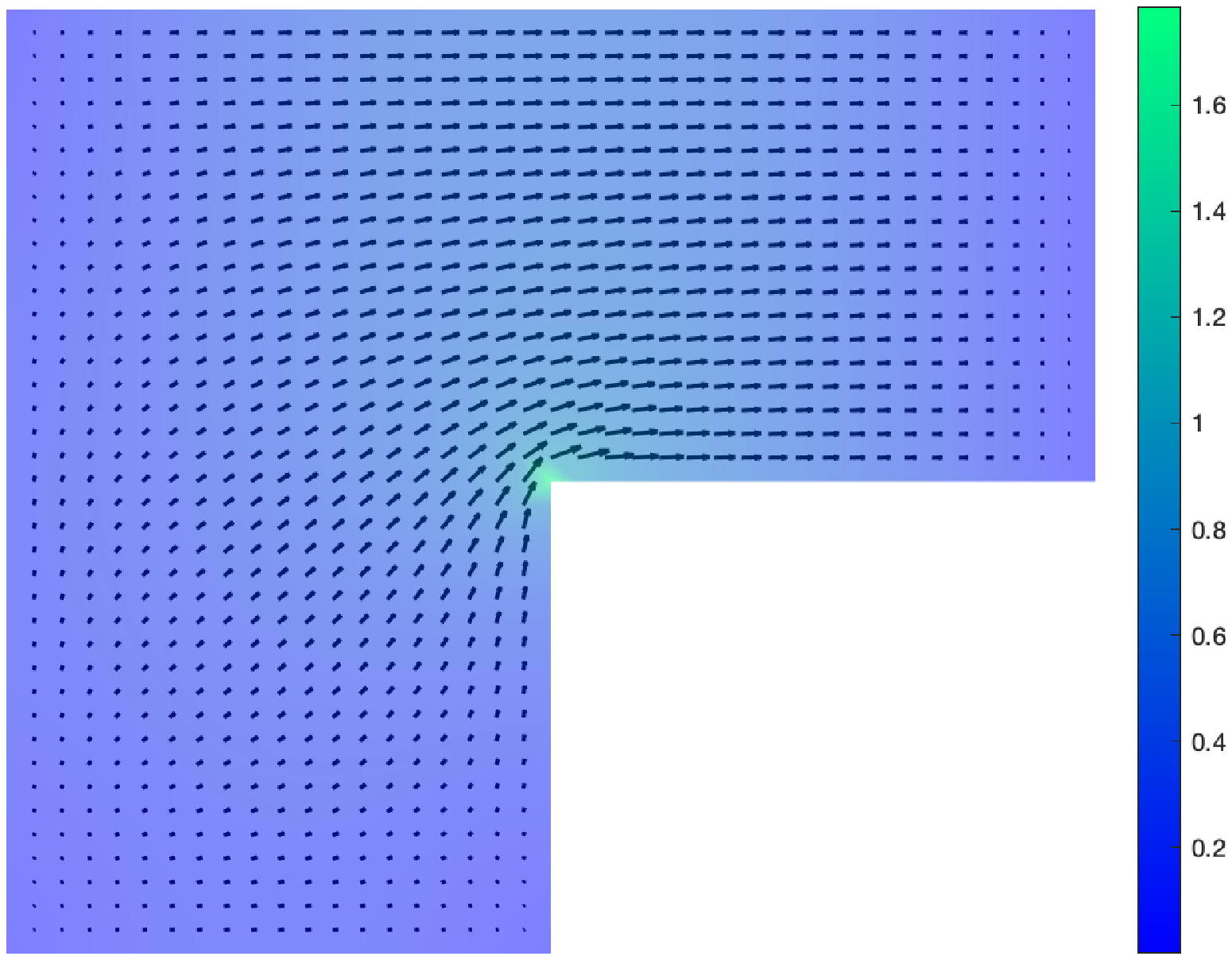}
	\end{minipage}}
	\hfill 
  \subfloat[Scheme \ref{algo3}]{
	\begin{minipage}[c][0.8\width]{
	   0.4\textwidth}
	   \centering
	   \includegraphics[width=1\textwidth]{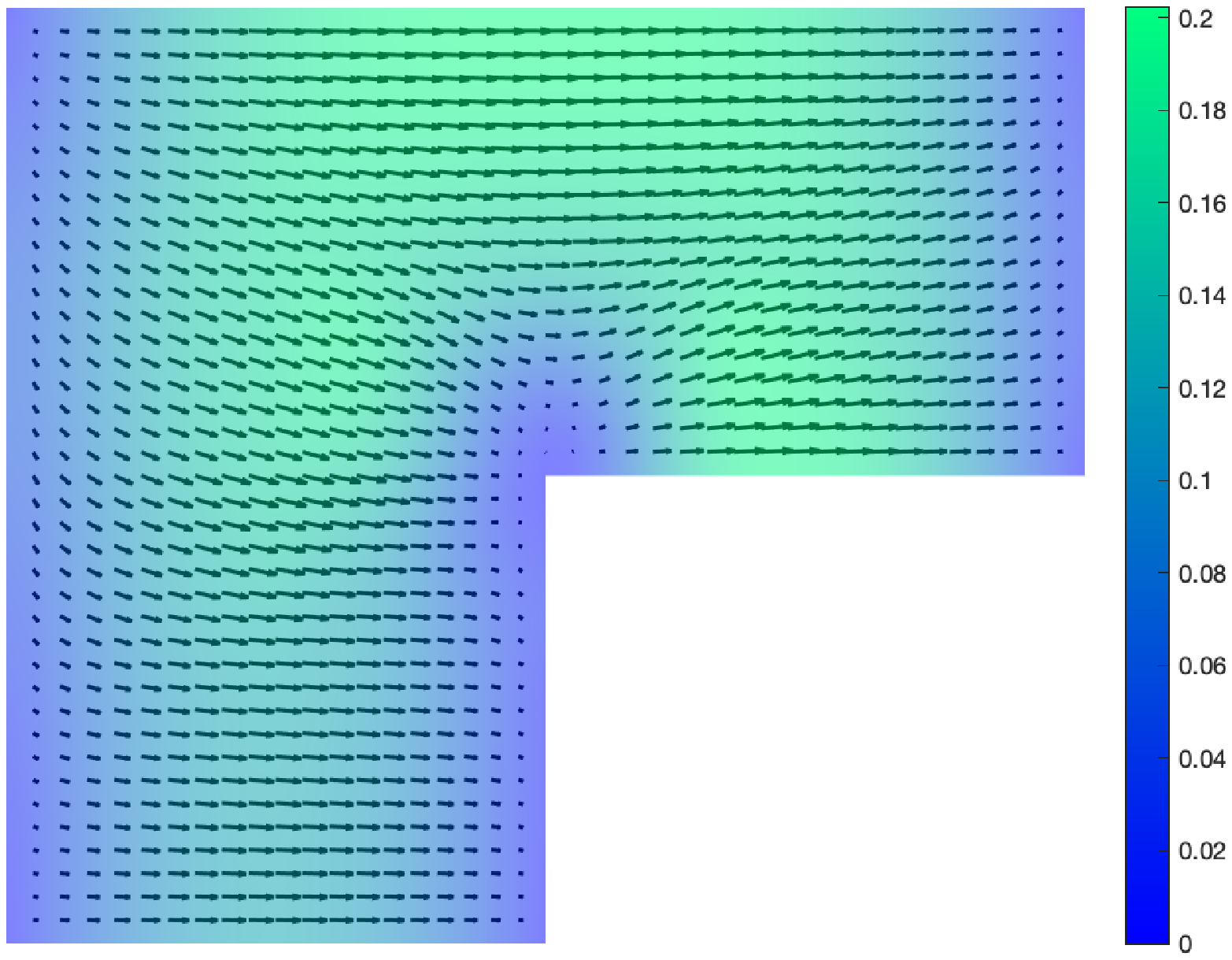}
	\end{minipage}}
	\hfill 
	\subfloat[Scheme \ref{algo4}]{
	\begin{minipage}[c][0.8\width]{
	   0.4\textwidth}
	   \centering
	   \includegraphics[width=1\textwidth]{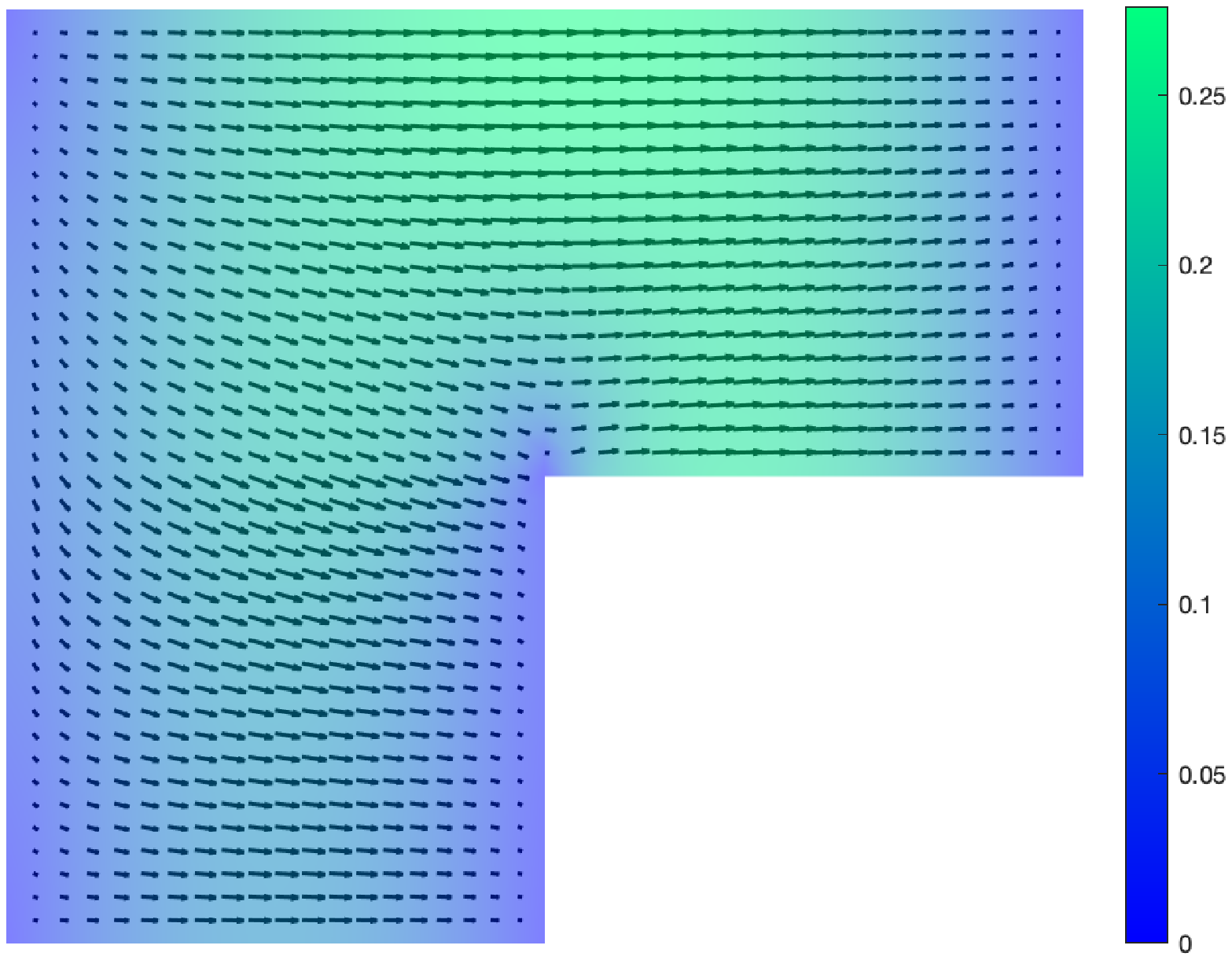}
	\end{minipage}}
	\hfill 
 \caption{Numerical solutions to the problem \eqref{prob-source1} on an L-shape domain with $\bm f=(1,0)^T$.}\label{fig-source}
\end{figure}

The numerical results for the Hodge-Laplacian source problem \eqref{prob-source1} with $\bm f=(1,0)^{\mathrm T}$ on an L-shape domain are shown in Figure \ref{fig-source}.
As we see in Figure \ref{fig-source}, the primal formulation with the Argyris and the $H^{1}(\rot)$-conforming finite elements show different solutions compared with the mixed formulations with the $H(\grad\rot)$- and $H^{1}(\rot)$-conforming elements. In fact, the primal formulation produces spurious solutions, which we will elaborate on in Section 4 by proving the convergence of the mixed formulations and explaining the reason of the spurious solutions.

\subsection{Eigenvalue problem}
Similar to the source problem, we consider the following four numerical schemes for the problem \eqref{prob-eig}.
\begin{scheme}[Mixed formulation with the $H(\grad\rot)$-conforming element]\label{algo5}
Find $(\lambda_h,\bm u_h,\sigma_{h}) \in  \mathbb R\times V_h\times S_{h}$,  such that
\begin{equation*}
\begin{split}
(\grad\rot\bm u_h,\grad\rot\bm v_h)+(\grad\sigma_h,\bm v_h)&= \lambda_h(\bm u_h,\bm v_h), \ \forall \bm v_h\in V_h,\\
(\bm u_h,\grad \tau_h)-(\sigma_h,\tau_h)&=0,\ \forall \tau_h\in S_h.
\end{split}
\end{equation*}
\end{scheme}
\begin{scheme}[Mixed formulation with the $H^1(\rot)$-conforming element]\label{algo6}
Find $(\lambda_h,\bm u_h,\sigma_{h}) \in \mathbb R \times V_h^1\times S_{h}^1$,  such that
\begin{equation*}
\begin{split}
(\grad\rot\bm u_h,\grad\rot\bm v_h)+(\grad\sigma_h,\bm v_h)&=\lambda_h(\bm u_h,\bm v_h), \ \forall \bm v_h\in V_h^1,\\
(\bm u_h,\grad \tau_h)-(\sigma_h,\tau_h)&=0,\ \forall \tau_h\in S_h^1.
\end{split}
\end{equation*}
\end{scheme}
\begin{scheme}[Primal formulation with the $H^1(\rot)$-conforming element]\label{algo7}
Find $(\lambda_h,\bm u_h) \in \mathbb R\times \mathring V_h^{1}$,  such that
\begin{equation*}
\begin{split}
(\grad\rot\bm u_h,\grad\rot\bm v_h)+(\div\bm u_h,\div\bm v_h)&=\lambda_h(\bm u_h,\bm v_h),\quad \forall \bm v_h\in \mathring V_h^{1}.
\end{split}
\end{equation*}
\end{scheme}
\begin{scheme}[Primal formulation with the $H^2$-conforming (Argyris) element]\label{algo8}
	Find $(\lambda_h,\bm u_h) \in \mathbb R\times \mathring V_h^{Arg}$,  such that
\begin{equation*}
\begin{split}
(\grad\rot\bm u_h,\grad\rot\bm v_h)+(\div\bm u_h,\div\bm v_h)
&=\lambda_h(\bm u_h,\bm v_h),\quad \forall \bm v_h\in \mathring V_h^{Arg}.\end{split}
\end{equation*}
\end{scheme}




We apply Schemes \ref{algo5} -- \ref{algo8} to solve the eigenvalue problem \eqref{prob-eig} on three different domains (see Figure \ref{fig1}):
\begin{itemize}
    \item  $\Omega_1=(0,1)\times (0,1)$.
    \item  $\Omega_2=(0,1) \times (0,1)\slash (1/3,3/4)\times(1/4,2/3)$.
    \item  $\Omega_3=(-1,1) \times (-1,1)\slash (0,1)\times(-1,0)$.
\end{itemize}
\begin{figure}[!ht]
	\centering
	\includegraphics[width=0.3\linewidth]{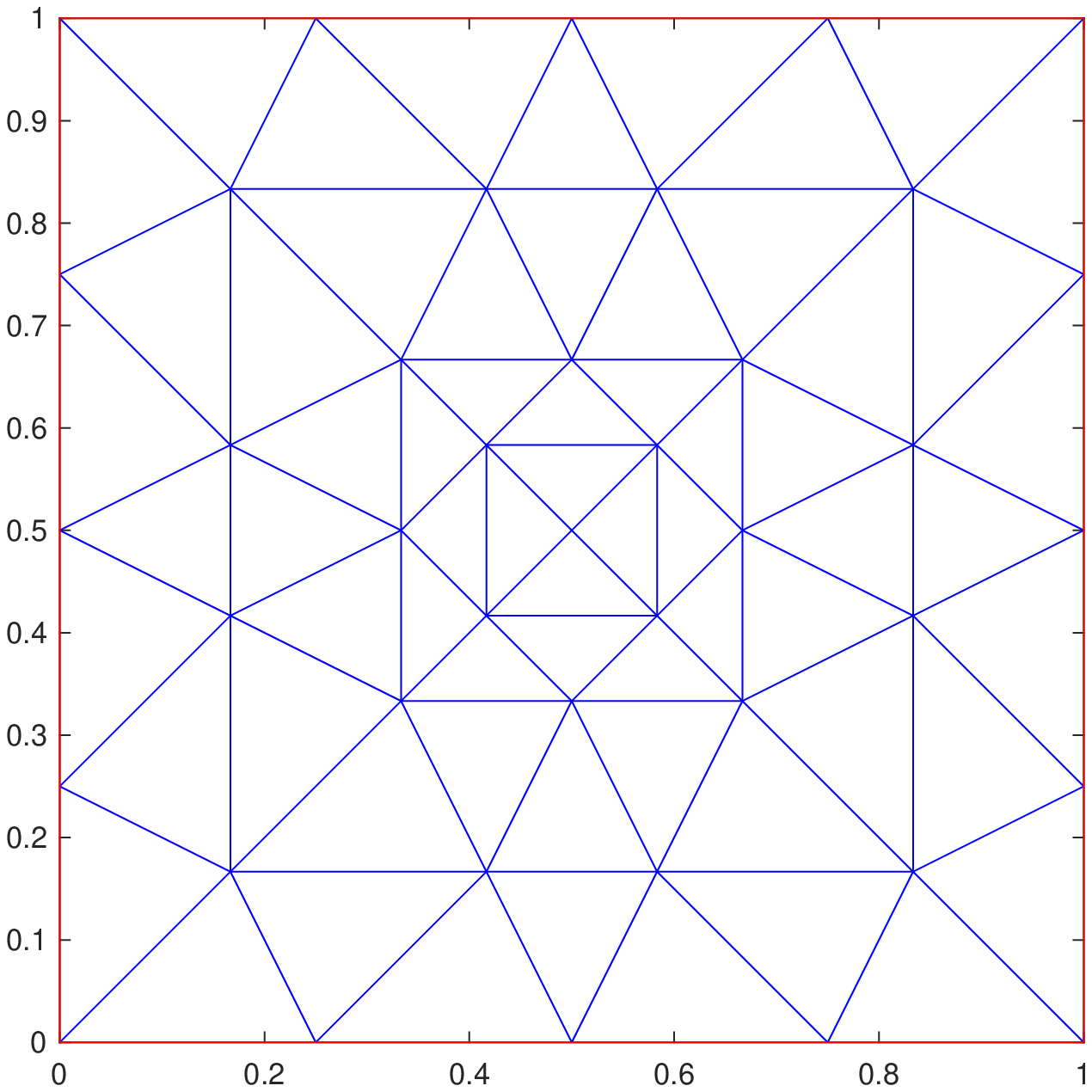}
	\includegraphics[width=0.3\linewidth]{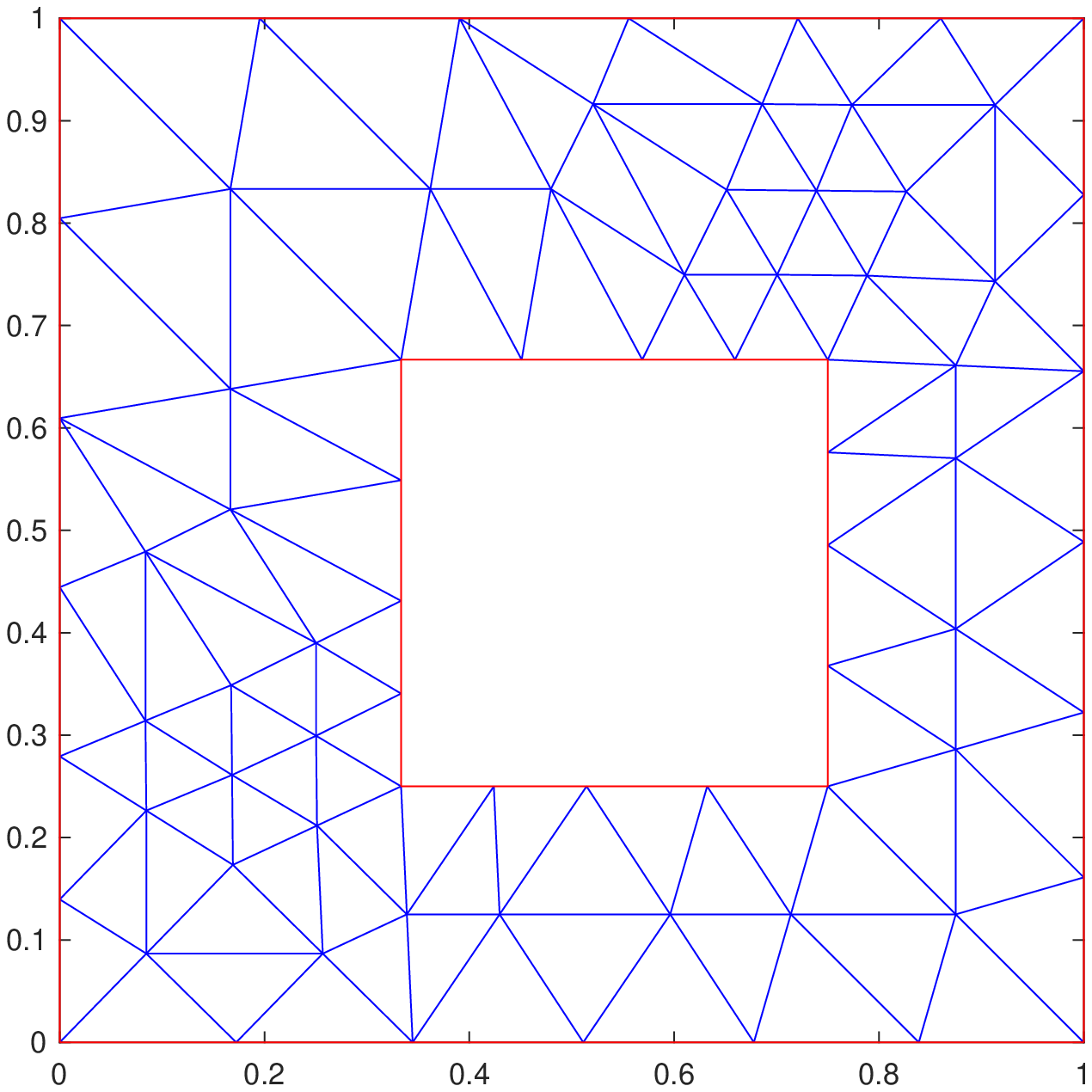}
		\includegraphics[width=0.3\linewidth]{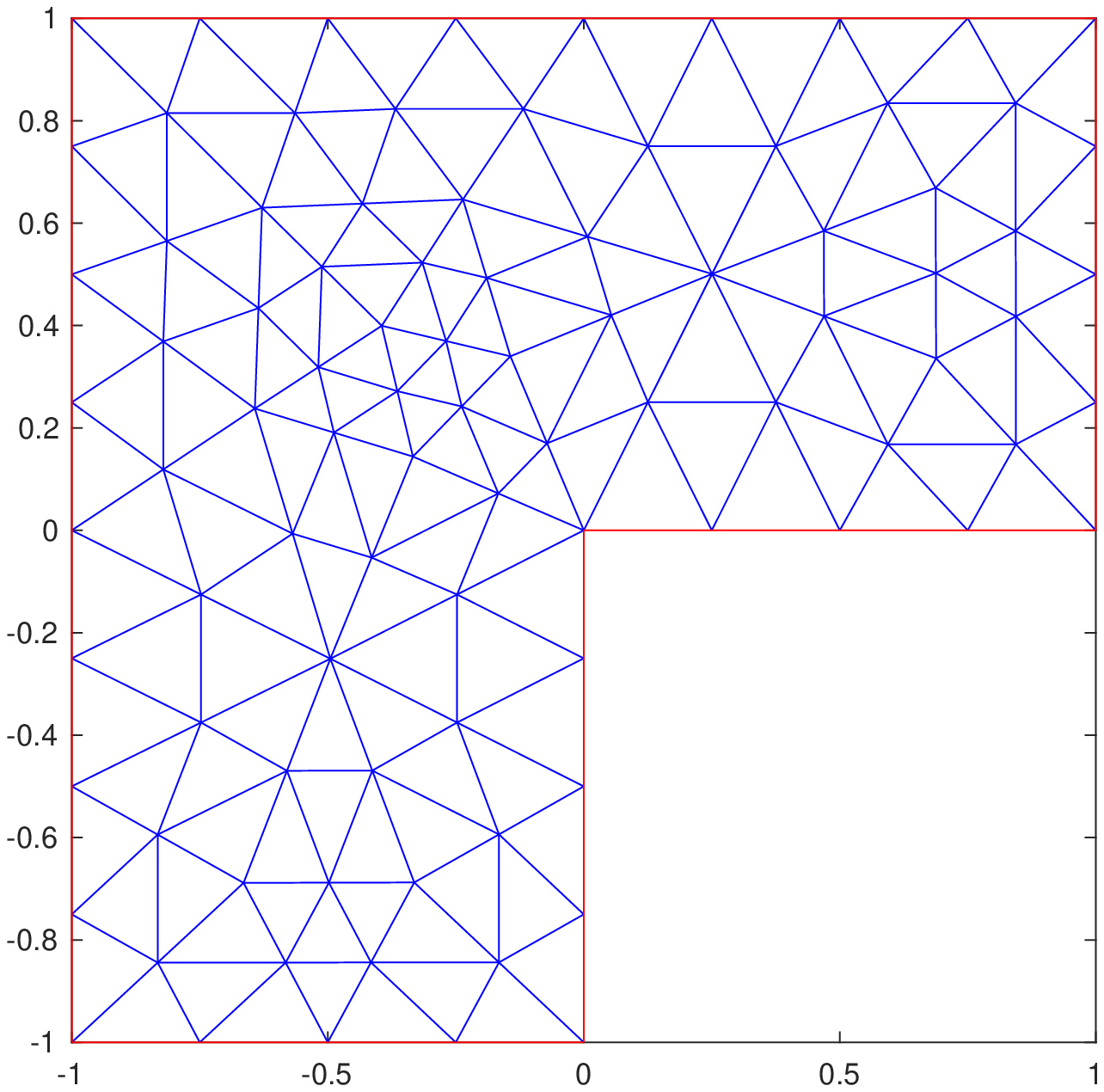}
	\caption{Initial meshes (refinement index $n=0$) for $\Omega_1$, $\Omega_2$, and $\Omega_3$}
	\label{fig1}
\end{figure}
We observe from Table \ref{tab-gradrot-omega1}--\ref{tab-argyris-omega3} that  the four schemes lead to the same numerical eigenvalues on $\Omega_1$ and different numerical eigenvalues on $\Omega_2$ and $\Omega_3$. 
We will prove in Section 4 that Scheme \ref{algo5} produces correctly convergent numerical eigenvalues on simply-connected domains, which implies that Scheme \ref{algo7} and Scheme \ref{algo8}  lead to spurious eigenvalues on $\Omega_3$.

\begin{table}[!ht]
	\centering
	\caption{Numerical eigenvalues with units $\pi^2$ on $\Omega_1$ obtained by {\bf Scheme \ref{algo5}} with $k=4$ for \eqref{prob-eig}} \label{tab-gradrot-omega1}
	    \begin{adjustwidth}{0cm}{0cm}
  {
	\begin{tabular}{ccccccccccc}
		\hline
		$n$&$\lambda_1$&$\lambda_2$&$\lambda_3$&$\lambda_4$&$\lambda_5$&$\lambda_6$&$\lambda_7$&$\lambda_8$\\ \hline
0&1.000000&   1.000000&   2.000000&   4.000001&   4.000001&   5.000002&   5.000002&   8.000011\\
1&1.000000&   1.000000&   2.000000&   4.000000&   4.000000&   5.000000&   5.000000&   8.000000\\
2&1.000000&   1.000000&   2.000000&   4.000000&   4.000000&   5.000000&   5.000000&   8.000000\\
      		\hline
	\end{tabular}
	    }
    \end{adjustwidth}
\end{table}  

\begin{table}[!ht]
	\centering
	\caption{Numerical eigenvalues with units $\pi^2$ on $\Omega_1$ obtained by {\bf Scheme \ref{algo6}} for \eqref{prob-eig}} \label{tab-mixed-omega1}
		    \begin{adjustwidth}{0cm}{0cm}
  {
	\begin{tabular}{ccccccccccc}
		\hline
		$n$&$\lambda_1$&$\lambda_2$&$\lambda_3$&$\lambda_4$&$\lambda_5$&$\lambda_6$&$\lambda_7$&$\lambda_8$\\
0&1.000000&   1.000000&   2.000000&   4.000000&   4.000000&   5.000000&   5.000000&   8.000001\\
1&1.000000&   1.000000&   2.000000&   4.000000&   4.000000&   5.000000&   5.000000&   8.000000\\
2&1.000000&   1.000000&   2.000000&   4.000000&   4.000000&   5.000000&   5.000000&   8.000000\\
         		\hline
	\end{tabular} }
    \end{adjustwidth}
\end{table}

\begin{table}[!ht]
	\centering
	\caption{Numerical eigenvalues with units $\pi^2$ on $\Omega_1$ obtained by {\bf Scheme \ref{algo7}} for \eqref{prob-eig}} \label{tab-primal-omega1}
			    \begin{adjustwidth}{0cm}{0cm}
  {
	\begin{tabular}{ccccccccccc}
		\hline
		$n$&$\lambda_1$&$\lambda_2$&$\lambda_3$&$\lambda_4$&$\lambda_5$&$\lambda_6$&$\lambda_7$&$\lambda_8$\\
0&1.000000&   1.000000&   2.000000&   4.000001&   4.000001&   5.000004&   5.000004&   8.000077\\
1&1.000000&   1.000000&   2.000000&   4.000000&   4.000000&   5.000000&   5.000000&   8.000000\\
2&1.000000&   1.000000&   2.000000&   4.000000&   4.000000&   5.000000&   5.000000&   8.000000\\
   		\hline
	\end{tabular}}
    \end{adjustwidth}
\end{table}

\begin{table}[!ht]
	\centering
	\caption{Numerical eigenvalues with units $\pi^2$ on $\Omega_1$ obtained by {\bf Scheme \ref{algo8}} for \eqref{prob-eig}} \label{tab-argyris-omega1}
			    \begin{adjustwidth}{0cm}{0cm}
  {\begin{tabular}{ccccccccc}
		\hline
		$n$&$\lambda_1$&$\lambda_2$&$\lambda_3$&$\lambda_4$&$\lambda_5$&$\lambda_6$&$\lambda_7$&$\lambda_8$\\
0&1.000000&   1.000000&   2.000000&   4.000000&   4.000000&   5.000000&   5.000000&   8.000003\\
1&1.000000&   1.000000&   2.000000&   4.000000&   4.000000&   5.000000&   5.000000&   8.000000\\
2&1.000000&   1.000000&   2.000000&   4.000000&   4.000000&   5.000000&   5.000000&   8.000000\\
   		\hline
	\end{tabular}}
    \end{adjustwidth}
\end{table}


\subsection{Eigenvalue problem with different boundary conditions}
Denote 
\[X=H(\grad\rot;\Omega)\cap H_0(\div;\Omega).\]
In this section, we consider another eigenvalue problem: find $(\lambda, \bm u)\in \mathbb R\times X$, 
such that 
\begin{equation}\label{prob2}
\begin{split}
-\curl\Delta\rot\bm u-\grad\div\bm u&=\lambda \bm u\ \text{in}\;\Omega,
\end{split}
\end{equation}
with the boundary conditions 
\begin{align*}
\Delta\rot\bm u=0,\ 
\grad\rot\bm u\cdot\bm n=0, \text{ and }
\bm u\cdot\bm n=0\ \text{on}\;\partial \Omega.
\end{align*}


The  primal variational formulation is to find $(\lambda, \bm u)\in \mathbb R\times X$ s.t.
\begin{equation}
\begin{split}
(\grad\rot\bm u,\grad\rot\bm v)+(\div\bm u,\div\bm v)&=\lambda (\bm u,\bm v), \ \forall \bm v\in X.
\end{split}
\end{equation}
The mixed variational formulation is to find $(\lambda, \bm u,\sigma)\in \mathbb R$ $\times H(\grad\rot;\Omega)$ $\times H^1(\Omega)$ s.t.
\begin{equation}
\begin{split}
(\grad\rot\bm u,\grad\rot\bm v)+(\grad\sigma,\bm v)&=\lambda (\bm u,\bm v), \ \forall \bm v\in H(\grad\rot;\Omega),\\
(\bm u,\grad \tau)-(\sigma,\tau)&=0,\ \forall \tau\in H^1(\Omega).
\end{split}
\end{equation}
We consider four numerical schemes similar to Scheme \ref{algo5}--\ref{algo8} but without the boundary condition $\rot\bm u_h=0$ on the three different domains. 
Again, we observe that the four schemes lead to different numerical solutions. In particular, from the mixed formulations we obtain one zero eigenvalue on $\Omega_{1}$, $\Omega_{3}$, and two zero eigenvalues on $\Omega_{2}$. On the other hand, other schemes do not produce zero numerical eigenvalues.  This difference will also be explained in Section 4.

\begin{table}[!ht]
	\centering
	\caption{Numerical eigenvalues with units $\pi^2$ on $\Omega_2$ obtained by {\bf Scheme \ref{algo5}}  with $k=4$ for \eqref{prob-eig}} \label{tab-gradrot-omega2}
			    \begin{adjustwidth}{0cm}{0cm}
  {\begin{tabular}{ccccccccccc}
		\hline
		$n$&$\lambda_1$&$\lambda_2$&$\lambda_3$&$\lambda_4$&$\lambda_5$&$\lambda_6$&$\lambda_7$&$\lambda_8$\\
0& 0.000000& 0.594212& 0.595733& 1.802009& 2.843750&  4.460286& 4.495899& 5.463774\\
1& 0.000000& 0.593616& 0.595336& 1.801970& 2.839489&  4.458673& 4.493247& 5.463500\\
2& 0.000000& 0.593379& 0.595179& 1.801959& 2.837796&  4.458048& 4.492200& 5.463407\\      		\hline
	\end{tabular}}
    \end{adjustwidth}
\end{table}   

\begin{table}[!ht]
	\centering
	\caption{Numerical eigenvalues with units $\pi^2$ on $\Omega_2$ obtained by {\bf Scheme \ref{algo6}} for \eqref{prob-eig}} \label{tab-mixed-omega2}
			    \begin{adjustwidth}{0cm}{0cm}
  {\begin{tabular}{ccccccccccc}
		\hline
		$n$&$\lambda_1$&$\lambda_2$&$\lambda_3$&$\lambda_4$&$\lambda_5$&$\lambda_6$&$\lambda_7$&$\lambda_8$\\
0& 0.000000& 0.596944& 0.597564& 1.802182& 2.863546&  4.467690& 4.508052& 5.465057\\
1& 0.000000& 0.594698& 0.596060& 1.802016& 2.847332&  4.461544& 4.498047& 5.463929\\
2& 0.000000& 0.593808& 0.595466& 1.801975& 2.840909&  4.459177& 4.494100& 5.463565\\
\hline
	\end{tabular}}
    \end{adjustwidth}
\end{table}

\begin{table}[!ht]
	\centering
	\caption{Numerical eigenvalues with units $\pi^2$ on $\Omega_2$ obtained by {\bf Scheme \ref{algo7}} for \eqref{prob-eig}} \label{tab-primal-omega2}
			    \begin{adjustwidth}{0cm}{0cm}
  {\begin{tabular}{ccccccccccc}
		\hline
		$n$&$\lambda_1$&$\lambda_2$&$\lambda_3$&$\lambda_4$&$\lambda_5$&$\lambda_6$&$\lambda_7$&$\lambda_8$\\
0& 2.645076 &3.202686 &   3.607223 &   4.369787 &   6.145767 &   7.964110 &   8.167482 &8.213072 \\
  1& 2.269742 &2.874862 &   3.141026 &   4.063906 &   5.846892 &   7.677691 &   7.894476 &7.971607 \\
  2& 2.065438 &2.689171 &   2.882076 &   3.886972 &   5.659409 &   7.489803 &   7.694887 &7.864160 \\
  3& 1.947637 &2.579732 &   2.731537 &   3.781333 &   5.542562 &   7.373284 &   7.571471 &7.797919 \\
   		\hline
	\end{tabular}}
    \end{adjustwidth}
\end{table}

\begin{table}[!ht]
	\centering
	\caption{Numerical eigenvalues with units $\pi^2$ on $\Omega_2$ obtained by {\bf Scheme \ref{algo8}} for \eqref{prob-eig}} \label{tab-argyris-omega2}
			    \begin{adjustwidth}{0cm}{0cm}
  {\begin{tabular}{ccccccccc}
		\hline
		$n$&$\lambda_1$&$\lambda_2$&$\lambda_3$&$\lambda_4$&$\lambda_5$&$\lambda_6$&$\lambda_7$&$\lambda_8$\\
0&2.996956 &   3.517662 &   3.946885 &   4.626899 &   6.440491 &   8.154377 &   8.449486 &   8.484641 \\
   1&2.462447 &   3.063195 &   3.349072 &   4.229950 &   6.010834 &   7.860474 &   8.069506 &   8.098664 \\
   2&2.178163 &   2.802148 &   3.008808 &   3.990078 &   5.760831 &   7.614372 &   7.802093 &   7.935650 \\
   3&2.015498 &   2.648156 &   2.808560 &   3.844587 &   5.606347 &   7.451910 &   7.636857 &   7.841693 \\
   		\hline
	\end{tabular}}
    \end{adjustwidth}
\end{table}

\begin{table}[!ht]
	\centering
	\caption{Numerical eigenvalues  with units $\pi^2$ on $\Omega_3$ obtained by {\bf Scheme \ref{algo5}} with $k=4$ for \eqref{prob-eig}} \label{tab-gradrot-omega3}
			    \begin{adjustwidth}{0cm}{0cm}
  {\begin{tabular}{ccccccccccc}
		\hline
		$n$&$\lambda_1$&$\lambda_2$&$\lambda_3$&$\lambda_4$&$\lambda_5$&$\lambda_6$&$\lambda_7$&$\lambda_8$\\
0&0.149678&   0.358073&   1.000000&   1.000000&   1.153997&   1.274383&   2.000000&   2.172031\\
1&0.149578&   0.358072&   1.000000&   1.000000&   1.153996&   1.274062&   2.000000&   2.171278\\
2&0.149538&   0.358072&   1.000000&   1.000000&   1.153996&   1.273934&   2.000000&   2.170978\\  
      		\hline
	\end{tabular}}
    \end{adjustwidth}
\end{table}   

\begin{table}[!ht]
	\centering
	\caption{Numerical eigenvalues  with units $\pi^2$ on $\Omega_3$ obtained by {\bf Scheme \ref{algo6}} for \eqref{prob-eig}} \label{tab-mixed-omega3}
			    \begin{adjustwidth}{0cm}{0cm}
  {\begin{tabular}{ccccccccccc}
		\hline
		$n$&$\lambda_1$&$\lambda_2$&$\lambda_3$&$\lambda_4$&$\lambda_5$&$\lambda_6$&$\lambda_7$&$\lambda_8$\\
0&0.150209&   0.358082&   1.000000&   1.000000 &  1.154010&   1.276086&   2.000000 &  2.176030\\
1&0.149788&   0.358074&   1.000000&   1.000000 &  1.153998&   1.274741&   2.000000 &  2.172873\\
2&0.149621&   0.358072&   1.000000&   1.000000 &  1.153996&   1.274203&   2.000000 &  2.171612\\
         		\hline
	\end{tabular}}
    \end{adjustwidth}
\end{table}

\begin{table}[!ht]
	\centering
	\caption{Numerical eigenvalues   with units $\pi^2$ on $\Omega_3$ obtained by {\bf Scheme \ref{algo7}} for \eqref{prob-eig}} \label{tab-primal-omega3}
			    \begin{adjustwidth}{0cm}{0cm}
  {\begin{tabular}{ccccccccccc}
		\hline
		$n$&$\lambda_1$&$\lambda_2$&$\lambda_3$&$\lambda_4$&$\lambda_5$&$\lambda_6$&$\lambda_7$&$\lambda_8$\\
0& 0.416285& 0.665296& 1.000000& 1.000000& 1.181067&  1.558700& 2.000000& 2.447851\\
1&  0.393230& 0.635841& 1.000000& 1.000000& 1.170019& 1.536785& 2.000000& 2.416211\\
2& 0.379644& 0.618841& 1.000000& 1.000000& 1.163726& 1.524513& 2.000000& 2.396893\\
   		\hline
	\end{tabular}}
    \end{adjustwidth}
\end{table}  

\begin{table}[!ht]
	\centering
	\caption{Numerical eigenvalues  with units $\pi^2$ on $\Omega_3$ obtained by {\bf Scheme \ref{algo8}} for \eqref{prob-eig}} \label{tab-argyris-omega3}
			    \begin{adjustwidth}{0cm}{0cm}
  {\begin{tabular}{ccccccccc}
		\hline
		$n$&$\lambda_1$&$\lambda_2$&$\lambda_3$&$\lambda_4$&$\lambda_5$&$\lambda_6$&$\lambda_7$&$\lambda_8$\\
0&0.431506&   0.667232&   1.000000&   1.000000&   1.185750&   1.559335&   2.000000&   2.471656\\
1&0.401863&   0.638780 &  1.000000&   1.000000&   1.173429&   1.538747&   2.000000&   2.429102\\
2&0.384768&   0.621402&   1.000000&   1.000000&   1.165924&   1.526300&   2.000000&   2.404445\\
   		\hline
	\end{tabular}}
    \end{adjustwidth}
\end{table}

 \begin{table}[!ht]
	\centering
	\caption{Numerical eigenvalues with units $\pi^2$ on $\Omega_1$ obtained by {\bf Scheme \ref{algo5}} with $k=4$ for \eqref{prob2}} \label{tab5}
\begin{adjustwidth}{0cm}{0cm}
  {\begin{tabular}{ccccccccc}
		\hline
		$n$&$\lambda_1$&$\lambda_2$&$\lambda_3$&$\lambda_4$&$\lambda_5$&$\lambda_6$&$\lambda_7$&$\lambda_8$\\
  $0$&0.000000  &  1.000000 &  1.000000 &  2.000000 &  4.000000 & 4.000001&5.000002&5.000002\\
  $1$&0.000000  &  1.000000 &  1.000000 &  2.000000 &  4.000000 & 4.000000&5.000000&5.000000\\
  $2$&-0.000000  &  1.000000 &  1.000000 &  2.000000 &  4.000000 & 4.000000&5.000000&5.000000\\
 		\hline
	\end{tabular}}
    \end{adjustwidth}
\end{table}

\begin{table}[!hb]
	\centering
	\caption{Numerical eigenvalues with units $\pi^2$ on $\Omega_1$ obtained by {\bf Scheme \ref{algo6}} for \eqref{prob2}} \label{tab5}
\begin{adjustwidth}{0cm}{0cm}
  {\begin{tabular}{ccccccccc}
		\hline
		$n$&$\lambda_1$&$\lambda_2$&$\lambda_3$&$\lambda_4$&$\lambda_5$&$\lambda_6$&$\lambda_7$&$\lambda_8$\\
	  $0$&0.000000  &  1.000000 &  1.000000 &  2.000000 &  4.000000 & 4.000000&5.0000001&5.0000001\\
  $1$&0.000000  &  1.000000 &  1.000000 &  2.000000 &  4.000000 & 4.000000&5.0000000&5.0000000\\
  $2$&-0.000000  &  1.000000 &  1.000000 &  2.000000 &  4.000000 & 4.000000&5.0000000&5.0000000\\

				\hline
	\end{tabular}}
    \end{adjustwidth}
\end{table}


\begin{table}[!hb]
	\centering
	\caption{Numerical eigenvalues with units $\pi^2$ on $\Omega_1$ obtained by {\bf Scheme \ref{algo7}} for \eqref{prob2}} \label{tab5}
\begin{adjustwidth}{0cm}{0cm}
  {\begin{tabular}{ccccccccc}
		\hline
		$n$&$\lambda_1$&$\lambda_2$&$\lambda_3$&$\lambda_4$&$\lambda_5$&$\lambda_6$&$\lambda_7$&$\lambda_8$\\
   0&1.000000&   1.000000&    2.000000&     4.000002 &   4.000002 &   5.000004  &  5.000004   & 7.762197 \\
   1&1.000000&   1.000000&    2.000000&     4.000000 &   4.000000 &   5.000000  &  5.000000   & 5.940076 \\
   2&1.000000&   1.000000&    2.000000&     4.000000 &   4.000000  &  4.833375   & 5.000000   & 5.000000 \\
   3&1.000000&   1.000000&    2.000000&     4.000000 &   4.000000 &   4.078384  &  5.000000   & 5.000000 \\				\hline
	\end{tabular}}
    \end{adjustwidth}
\end{table}

\begin{table}[!hb]
	\centering
	\caption{Numerical eigenvalues with units $\pi^2$ on $\Omega_1$ obtained by {\bf  Scheme \ref{algo8}} for \eqref{prob2}} \label{tab5}
				    \begin{adjustwidth}{0cm}{0cm}
  {\begin{tabular}{ccccccccc}
		\hline
		$n$&$\lambda_1$&$\lambda_2$&$\lambda_3$&$\lambda_4$&$\lambda_5$&$\lambda_6$&$\lambda_7$&$\lambda_8$\\
0&1.000000&   1.000000&   2.000000&   4.000000&   4.000000&   4.353529& 5.000000& 5.000000\\
1&1.000000&   1.000000&   2.000000&   3.732480&   4.000000&   4.000000& 5.000000& 5.000000\\
2&1.000000&   1.000000&   2.000000&   3.266372&   4.000000&   4.000000& 5.000000& 5.000000\\
3&1.000000&   1.000000&   2.000000&   2.903702&   4.000000&   4.000000& 5.000000& 5.000000\\
		\hline
	\end{tabular}}
    \end{adjustwidth}
\end{table}

 \begin{table}[!hb]
	\centering
	\caption{Numerical eigenvalues with units $\pi^2$ on $\Omega_2$ obtained by {\bf Scheme \ref{algo5}} with $k=4$ for \eqref{prob2}} \label{tab5}
				    \begin{adjustwidth}{0cm}{0cm}
  {\begin{tabular}{ccccccccc}
		\hline
		$n$&$\lambda_1$&$\lambda_2$&$\lambda_3$&$\lambda_4$&$\lambda_5$&$\lambda_6$&$\lambda_7$&$\lambda_8$\\
  $0$&0.000000  &  -0.000000 &  0.594212 &  0.595733 &  1.802009 & 2.843750&4.460286&4.495899\\
  $1$&0.000000&  -0.000000&  0.593616 &  0.595336 &  1.801970 &2.839489&4.458673&4.493248\\
  $2$&0.000000&  -0.000000 &  0.593379 & 0.595179 &  1.801960 & 2.837797&4.458049&4.492195\\
		\hline
	\end{tabular}}
    \end{adjustwidth}
\end{table}

\begin{table}[!hb]
	\centering
	\caption{Numerical eigenvalues with units $\pi^2$ on $\Omega_2$ obtained by {\bf Scheme \ref{algo6}} for \eqref{prob2}} \label{tab5}
				    \begin{adjustwidth}{0cm}{0cm}
  {\begin{tabular}{ccccccccc}
		\hline
		$n$&$\lambda_1$&$\lambda_2$&$\lambda_3$&$\lambda_4$&$\lambda_5$&$\lambda_6$&$\lambda_7$&$\lambda_8$\\
	    0&0.000000&-0.000000&0.596944&0.597564&1.802182&2.863546&4.467690&4.508052\\
		1&0.000000 &-0.000000&0.594698&0.596060&1.802016&2.847329&4.461542&4.498048\\
		2&0.000000&-0.000000&0.593808&0.595465&1.801975&2.840909&4.459231&4.494105\\
				\hline
	\end{tabular}}
    \end{adjustwidth}
\end{table}

\begin{table}[!hb]
	\centering
	\caption{Numerical eigenvalues with units $\pi^2$ on $\Omega_2$ obtained by {\bf  Scheme \ref{algo7}} for \eqref{prob2}} \label{tab5}
				    \begin{adjustwidth}{0cm}{0cm}
  {\begin{tabular}{ccccccccc}
		\hline
		$n$&$\lambda_1$&$\lambda_2$&$\lambda_3$&$\lambda_4$&$\lambda_5$&$\lambda_6$&$\lambda_7$&$\lambda_8$\\
	    0& 2.640174 &   3.189826 &   3.594658 &   4.335878 &   6.144269 &   7.950591 &   8.156921 &   8.201252 \\
  1& 2.267177 &   2.860417 &   3.128707 &   4.010749 &   5.845370 &   7.650719 &   7.875648 &   7.961485 \\
  2& 2.063909 &   2.673185 &   2.869503 &   3.813822 &   5.657945 &   7.452232 &   7.668946 &   7.852054 \\
  3& 1.946566 &   2.562122 &   2.718412 &   3.687886 &   5.541129 &   7.326460 &   7.539070 &   7.783510 \\
				\hline
	\end{tabular}}
    \end{adjustwidth}
\end{table}

\begin{table}[!hb]
	\centering
	\caption{Numerical eigenvalues with units $\pi^2$ on $\Omega_2$ obtained by {\bf  Scheme \ref{algo8}} for \eqref{prob2}} \label{tab5}
				    \begin{adjustwidth}{0cm}{0cm}
  {\begin{tabular}{ccccccccc}
		\hline
		$n$&$\lambda_1$&$\lambda_2$&$\lambda_3$&$\lambda_4$&$\lambda_5$&$\lambda_6$&$\lambda_7$&$\lambda_8$\\
  0&  2.987390&   3.489137 &   3.927939 &   4.570472 &   6.438141 &   8.140561 &   8.420141 &   8.469386 \\
 1&  2.457624&   3.037089 &   3.331364 &   4.148755 &   6.008541 &   7.823479 &   8.045921 &   8.078550 \\
  2& 2.175539&   2.777434 &   2.992092 &   3.886997 &   5.758800 &   7.562974 &   7.770874 &   7.918637 \\
  3& 2.013865&   2.623800 &   2.792272 &   3.720863 &   5.604514 &   7.391764 &   7.599227 &   7.823772 \\
		\hline
	\end{tabular}}
    \end{adjustwidth}
\end{table}


 \begin{table}[!hb]
	\centering
	\caption{Numerical eigenvalues with units $\pi^2$ on $\Omega_3$ obtained by {\bf Scheme \ref{algo5}} with $k=4$ for \eqref{prob2}} \label{tab5}
				    \begin{adjustwidth}{0cm}{0cm}
  {\begin{tabular}{ccccccccc}
		\hline
		$n$&$\lambda_1$&$\lambda_2$&$\lambda_3$&$\lambda_4$&$\lambda_5$&$\lambda_6$&$\lambda_7$&$\lambda_8$\\
0&   0.000000&   0.149678 &   0.358073 &   1.000000&   1.000000  & 1.153996 &   1.274383 &   2.000000\\
1&   0.000000&   0.149578 &   0.358072 &   1.000000&   1.000000  & 1.153996 &   1.274062 &   2.000000\\
2&   0.000000&   0.149538 &   0.358072 &   1.000000&   1.000000  & 1.153996 &   1.273934 &   2.000000\\
		\hline
	\end{tabular}}
    \end{adjustwidth}
\end{table}

\begin{table}[!hb]
	\centering
	\caption{Numerical eigenvalues with units $\pi^2$ on $\Omega_3$ obtained by {\bf Scheme \ref{algo6}} for \eqref{prob2}} \label{tab5}
				    \begin{adjustwidth}{0cm}{0cm}
  {\begin{tabular}{ccccccccc}
		\hline
		$n$&$\lambda_1$&$\lambda_2$&$\lambda_3$&$\lambda_4$&$\lambda_5$&$\lambda_6$&$\lambda_7$&$\lambda_8$\\
	 0&   0.000000 &   0.150209  &   0.358082  &   1.000000 &   1.000000 &  1.154010 &   1.276086 &   2.000000\\
1&   0.000000 &   0.149788  &   0.358074  &   1.000000 &   1.000000  & 1.153998 &   1.274741 &   2.000000\\
2 & -0.000000 &   0.149621  &   0.358072  &   1.000000 &   1.000000 &  1.153996 &   1.274203 &   2.000000\\
				\hline
	\end{tabular}}
    \end{adjustwidth}
\end{table}

\begin{table}[!hb]
	\centering
	\caption{Numerical eigenvalues with units $\pi^2$ on $\Omega_3$ obtained by {\bf Scheme \ref{algo7}} for \eqref{prob2}} \label{tab5}
				    \begin{adjustwidth}{0cm}{0cm}
  {\begin{tabular}{ccccccccc}
		\hline
		$n$&$\lambda_1$&$\lambda_2$&$\lambda_3$&$\lambda_4$&$\lambda_5$&$\lambda_6$&$\lambda_7$&$\lambda_8$\\
	   0& 0.415525 &   0.607103 &  1.000000 &   1.000000 &   1.180641 &   1.471030 &   2.000000 &   2.203868 \\
   1&0.392873 &   0.556247 &  1.000000 &   1.000000 &   1.169824 &   1.399646 &   1.989572 &   2.000000 \\
   2&0.379489 &   0.518223 &  1.000000 &   1.000000 &   1.163642 &   1.330472 &   1.856537 &   2.000000 \\
   3&0.371374 &   0.487486 &  1.000000 &   1.000000 &   1.159955 &   1.265016 &   1.776840 &   2.000000 \\				\hline
	\end{tabular}}
    \end{adjustwidth}
\end{table}

\begin{table}[!ht]
	\centering
	\caption{Numerical eigenvalues with units $\pi^2$ on $\Omega_3$ obtained by {\bf Scheme \ref{algo8}} for \eqref{prob2}} \label{tab5}
				    \begin{adjustwidth}{0cm}{0cm}
  {\begin{tabular}{ccccccccc}
		\hline
		$n$&$\lambda_1$&$\lambda_2$&$\lambda_3$&$\lambda_4$&$\lambda_5$&$\lambda_6$&$\lambda_7$&$\lambda_8$\\
   0&0.429946 &   0.569639 &   1.000000 &   1.000000 &   1.184886 &   1.369446 &   1.913935 &   2.000000 \\
   1&0.401203 &   0.520021 &   1.000000 &   1.000000 &   1.173065 &   1.292529 &   1.813290 &   2.000000 \\
   2&0.384498 &   0.482939 &   1.000000 &   1.000000 &   1.165775 &   1.227369 &   1.751382 &   2.000000 \\
   3&0.374438 &   0.453274 &   1.000000 &   1.000000 &   1.161302 &   1.172738 &   1.711802 &   2.000000 \\   
		\hline
	\end{tabular}}
    \end{adjustwidth}
\end{table}

\section{$\grad\rot$ complex and the Hodge-Laplacian problems}
To investigate the spurious solutions, in this section, we present the $\grad\rot$ complex in \cite{arnold2020complexes} and its connections to the problems considered  in Section \ref{experiments}. 
\subsection{$\grad\rot$ complex and cohomology}
For any real number $q$, the {\it $\grad\rot$ complex} in \cite{arnold2020complexes} reads:
\begin{equation}\label{grad-curl2D}
\begin{tikzcd}[column sep=1.5em]
0 \arrow{r}{}&H^{q}(\Omega)\arrow{r}{\grad} & H^{q-1}(\Omega)\otimes \mathbb V  \arrow{r}{\grad\rot} &[12] H^{q-3}(\Omega)\otimes \mathbb V  \arrow{r}{\rot} & H^{q-4}(\Omega) \arrow{r}{} & 0.
\end{tikzcd}
\end{equation}
The sequence \eqref{grad-curl2D} is derived by connecting two de~Rham complexes, more precisely, from the following diagram:
\begin{equation}\label{Z-complex-2D}
\begin{tikzcd}[column sep=1.6em]
0 \arrow{r} &H^{q}(\Omega)\arrow{r}{\grad} &H^{q-1}(\Omega)\otimes \mathbb V  \arrow{r}{\rot}&H^{q-2}(\Omega)  \arrow{r}{} & 0 \arrow{r}{} & 0\\
0 \arrow{r}&0 \arrow{r} \arrow[ur, "0"] &H^{q-2}(\Omega)\arrow{r}{\grad} \arrow[ur, "\mathrm{id}"]& H^{q-3}(\Omega) \otimes \mathbb V  \arrow{r}{\rot} \arrow[ur, "0"]&H^{q-4}(\Omega) \arrow{r}{}\arrow[ur, "0"]  & 0.
 \end{tikzcd}
\end{equation}
A major conclusion of the construction in \cite{arnold2020complexes} is that the cohomology of \eqref{grad-curl2D} is isomorphic to the cohomology of the rows of \eqref{Z-complex-2D}. Note that the dimension of cohomology at $H^{q-1}(\Omega)\otimes \mathbb V$ in the first row of \eqref{Z-complex-2D} is the first Betti number $\beta_{1}$ of the domain, and the dimension of cohomology at $H^{q-2}(\Omega)$ in the second row of \eqref{Z-complex-2D} is 1 (the kernel of $\grad$ consists of constants). Consequently,
$$
\ker (\grad\rot, H^{q-1}(\Omega)\otimes \mathbb V)= \ran (\grad, H^{q}(\Omega))\oplus\mathscr{H}_{\infty}, 
$$
where $\mathscr{H}_{\infty}$ is a finite dimensional space of smooth functions with dimension $\dim \mathscr{H}_{\infty}=1+ \beta_{1}$. 
\begin{remark}
A general algebraic construction from \eqref{Z-complex-2D} implies $\dim \mathscr{H}_{\infty}\leq 1+ \beta_1$ \cite[Theorem 6]{arnold2020complexes}. From the identity $\rot(u\bm x^{\perp})=\bm x\cdot \grad u+2u$, we see that $I=\rot \widetilde{K}_{1}-\widetilde{K}_{2}\grad$, where $\widetilde{K}_{1}u:=1/2u\bm x^{\perp}$ and $\widetilde{K}_{2}\bm v:=1/2\bm v\cdot \bm x$. Therefore one can verify the condition \cite[(28)]{arnold2020complexes} by defining $K_1,K_2$ in the way of \cite[(43)]{arnold2020complexes}. Then $\dim \mathscr{H}_{\infty}= 1+ \beta_1$ follows.
\end{remark}

The $L^{2}$ version of \eqref{grad-curl2D} with unbounded linear operators, i.e., 
\begin{equation}\label{L2-complex}
\begin{tikzcd}[column sep=1.5em]
0 \arrow{r}{}&L^{2}(\Omega)\arrow{r}{\grad} & L^{2}(\Omega)\otimes \mathbb V  \arrow{r}{\grad\rot} &[12] L^{2}(\Omega)\otimes \mathbb V  \arrow{r}{\rot} & L^{2}(\Omega) \arrow{r}{} & 0,
\end{tikzcd}
\end{equation}
is closely related to the PDEs and the numerics. Consider the following domain complex of \eqref{L2-complex} (c.f., \cite{arnold2020complexes})  
\begin{equation}\label{grad-curl-domain2D}
\begin{tikzcd}[column sep=1.5em]
0 \arrow{r}{}&H^{1}(\Omega)\arrow{r}{\grad} & H(\grad\rot;\Omega)  \arrow{r}{\grad\rot} &[12] H(\rot;\Omega) \arrow{r}{\rot} & L^{2}(\Omega) \arrow{r}{} & 0.
\end{tikzcd}
\end{equation}
Hereafter, for a differential operator $D$,
\[H(D;\Omega):=\{u\in L^{2}(\Omega) : Du\in L^{2}(\Omega)\}\] with the inner product 
\[(\bm u,\bm v)_{H(D; \Omega)}=(\bm u,\bm v)+(D\bm u,D\bm v)\]
and the norm \[\|\bm u\|^{2}_{H(D; \Omega)}=\|\bm u\|^{2}+\|D\bm u\|^{2}.\]

The domain complex of the adjoint of \eqref{grad-curl-domain2D} is
\begin{equation}
\begin{tikzcd}[column sep=1.5em]
0 & L_{0}^2(\Omega)\arrow{l}{} & H_0(\div;\Omega)  \arrow[l,"-\div"'] &[12]  H_{0}(\curl\div;\Omega) \arrow[l,"-\curl\div"'] & H_0^1(\Omega) \arrow[l,"\curl"'] &0\arrow{l}{}.
\end{tikzcd}
\end{equation}
Here $H_0(\div;\Omega)$ is the space of functions in $H(\div;\Omega)$ with vanishing trace and at this stage $H_{0}(\curl\div;\Omega)$ is a formal notation for the domain of the adjoint of the operator $(\grad\rot, H(\grad\rot;\Omega))$. We will characterize this space in Section \ref{characterization} to show that it is a subspace of $H(\curl\div;\Omega)$ with boundary conditions $\bm u\cdot\bm n=0$ and $\div\bm u=0$.

  A corollary of \cite[Theorem 1]{arnold2020complexes} is the following.
\begin{theorem}
The following decomposition holds
\begin{align}\label{decomp-gradcurl}
\ker (\grad\rot, H(\grad\rot;\Omega) )= \ran (\grad, H^{1}(\Omega))\oplus \mathscr{H}_{\infty},
\end{align}
with $\dim \mathscr{H}_{\infty}=1+ \beta_{1}$, where $\beta_{1}$ is the first Betti number of the domain.
\end{theorem}
From general results on Hilbert complexes, we have the Hodge decomposition
$$
L^{2}(\Omega)\otimes \mathbb V=\grad H^{1}(\Omega)\oplus \curl\div H_{0}(\curl\div;\Omega) \oplus \mathfrak H,
$$
where $\mathfrak H$ is the space of harmonic forms with the same dimension as $\mathscr{H}_{\infty}$. In addition to the harmonic forms of the de Rham complex, i.e. the functions satisfying $\rot\bm u=0$ and $\div\bm u=0$, the function $\bm u=\grad p$ with $p$ solving 
\[\Delta p=1 \text{ in } \Omega, \quad p=0\text{ on } \partial\Omega\]
is also a harmonic form in $\mathfrak H$. 
 
The Hodge-Laplacian operator follows from the abstract definition:
$$
\mathscr{L}:=-(\curl \div)\grad\rot-\grad\div=-\curl \Delta \rot-\grad\div,
$$
with the domain $D_{\mathscr{L}}=\{\bm u\in  X: \grad\rot \bm u\in  H_{0}(\curl\div;\Omega),\div \bm u\in H^1(\Omega) \}$.
 For $\bm f\in L^2(\Omega)\otimes  \mathbb V$, the strong formulation of the Hodge-Laplacian boundary value problem seeks $\bm u\in  D_{\mathscr{L}}$ and $\bm u\perp \mathfrak H$ such that
\begin{equation}\label{prob-source-withoutbc}
\begin{split}
-\curl \Delta \rot\bm u-\grad\div\bm u &=\bm f-P_{\mathfrak H}\bm f\quad \text{in}\;\Omega.\\
\end{split}
\end{equation}
Here and after, $P_{\mathfrak H}$ is the $L^2$ projection onto $\mathfrak H$.
The corresponding eigenvalue problem is to seek $(\lambda,\bm u)\in  \mathbb R\times D_{\mathscr{L}}$ such that
\begin{equation}\label{prob-source-withoutbc}
\begin{split}
-\curl \Delta \rot\bm u-\grad\div\bm u &=\lambda\bm u\quad \text{in}\;\Omega,
\end{split}
\end{equation}
which is exactly the eigenvalue problem \eqref{prob2}.

\subsection{Hodge-Laplacian boundary value problems}
Consider another domain complex of \eqref{L2-complex} with boundary conditions:
\begin{equation}\label{grad-curl-domain2D-bdc}
\begin{tikzcd}[column sep=1.5em]
0 \arrow{r}{}&H^{1}(\Omega)\arrow{r}{\grad} &  H_{\rot}(\grad\rot;\Omega)  \arrow{r}{\grad\rot} &[12] H_0(\rot;\Omega) \arrow{r}{\rot} & L^{2}_{0}(\Omega) \arrow{r}{} & 0,
\end{tikzcd}
\end{equation}
which is derived from the following diagram:
\begin{equation}\label{bgg-bc1}
\begin{tikzcd}
0 \arrow{r} &H^{1}(\Omega)\arrow{r}{\grad} &H_{\rot}(\grad\rot;\Omega)  \arrow{r}{\rot}&H^{1}_{0}(\Omega) \arrow{r}{} & 0 \arrow{r}{} & 0\\
0 \arrow{r}&0 \arrow{r} \arrow[ur, "0"] &H^{1}_{0}(\Omega)\arrow{r}{\grad} \arrow[ur, "\mathrm{id}"]& H_{0}(\rot;\Omega)  \arrow{r}{\rot} \arrow[ur, "0"]&L^{2}_{0}(\Omega) \arrow{r}{}\arrow[ur, "0"]  & 0.
 \end{tikzcd}
\end{equation}

By the algebraic construction in \cite{arnold2020complexes}, 
 we can show the cohomology of \eqref{grad-curl-domain2D-bdc} in a similar way as \eqref{grad-curl-domain2D}. In particular, we have 
$$
\ker (\grad\rot, H_{\rot}(\grad\rot;\Omega))=\ran (\grad, H^{1}(\Omega))\oplus \mathscr{H}_{\infty}^{\rot},
$$
where $\mathscr{H}_{\infty}^{\rot}$ is a set of smooth functions with dimension equal to the dimension of the cohomology of the first row of \eqref{bgg-bc1} at $H_{\rot}(\grad\rot;\Omega)$ plus  the dimension of the cohomology of the second row at $H_0^1(\Omega)$, and hence $\dim\mathscr{H}_{\infty}^{\rot}=\beta_1$, which means that there is no nontrivial cohomology on contractible domains.

The domain complex of the adjoint of \eqref{grad-curl-domain2D-bdc} is:
\begin{equation}
\begin{tikzcd}[column sep=1.5em]
0 & L_{0}^2(\Omega)\arrow{l}{} & H_0(\div;\Omega)  \arrow[l,"-\div"'] &[12]  H_{\div}(\curl\div;\Omega) \arrow[l,"-\curl\div"'] & H^1(\Omega) \arrow[l,"\curl"'] &0\arrow{l}{} .
\end{tikzcd}
\end{equation}
Here 
$
 H_{\div}(\curl\div;\Omega)
$
is again a formal notation for the domain of the adjoint of the unbounded operator $(\grad\rot, H_{\rot}(\grad\rot;\Omega))$. We will characterize this space in Section \ref{characterization} to show that it is a subspace $H(\curl\div;\Omega)$ with boundary condition $\div\bm u=0$. 

The Hodge decomposition at $H_{\rot}(\grad\rot)$ now reads:
\begin{equation}\label{Hodge}
L^{2}(\Omega)\otimes \mathbb V=\grad H^{1}(\Omega)\oplus \curl\div  H_{\div}(\curl\div;\Omega)\oplus\mathfrak H_{\rot},
\end{equation}
where the space of harmonic forms $\mathfrak H_{\rot}$ has the same dimension as $\mathscr{H}_{\infty}^{\rot}$, which is trivial if $\Omega$ is contractible.

The Hodge-Laplacian operator follows from the abstract definition:
$$
\mathscr{L}:=-(\curl \div)\grad\rot-\grad\div=-\curl \Delta \rot-\grad\div,
$$
with the domain $D_{\mathscr{L},\rot}=\{\bm u\in  X_{\rot} : \grad\rot \bm u\in  H_{\div}(\curl\div;\Omega),\div \bm u\in H^1(\Omega) \}$.
 For $\bm f\in L^2(\Omega)\otimes  \mathbb V$, the strong formulation of the Hodge-Laplacian boundary value problem seeks $\bm u\in  D_{\mathscr{L},\rot}$ and $\bm u\perp \mathfrak H_{\rot}$ such that
\begin{equation}\label{prob-source}
\begin{split}
-\curl \Delta \rot\bm u-\grad\div\bm u &=\bm f-P_{\mathfrak H_{\rot}}\bm f\quad \text{in}\;\Omega,\\
\end{split}
\end{equation}
The  primal variational formulation is to
seek $\bm u\in X_{\rot}$ and $\bm u\perp \mathfrak H_{\rot}$ such that
\begin{equation}\label{pvprob}
\begin{split}
(\grad\rot\bm u,\grad\rot\bm v)+(\div\bm u,\div\bm v)&=(\bm f-P_{\mathfrak H_{\rot}}\bm f,\bm v), \quad \forall \bm v\in X_{\rot}.
\end{split}
\end{equation}
The mixed variational formulation
seeks $(\bm u,\sigma)\in  H_{\rot}(\grad\rot;\Omega)\times H^1(\Omega)$, such that $\bm u\perp \mathfrak H_{\rot}$ and
\begin{equation}\label{mvprob}
\begin{split}
(\grad\rot\bm u,\grad\rot\bm v)+(\grad\sigma,\bm v)&= (\bm f-P_{\mathfrak H_{\rot}}\bm f,\bm v), \quad \forall \bm v\in  H_{\rot}(\grad\rot;\Omega),\\
(\bm u,\grad \tau)-(\sigma,\tau)&=0,\quad \forall \tau\in H^1(\Omega).
\end{split}
\end{equation}
\begin{remark}
When $\Omega$ is simply-connected, $\mathfrak H_{\rot}$ vanishes and hence the Hodge-Laplacian problem is exactly the problem \eqref{prob-source1}.
\end{remark}

According to Theorem 4.7 in \cite{arnold2018finite}, the strong formulation \eqref{prob-source}, the primal formulation \eqref{pvprob}, and the mixed formulation \eqref{mvprob} are equivalent. 
The well-posedness of  \eqref{prob-source} follows from standard results on the Hodge-Laplacian problems of Hilbert complexes (c.f., \cite[Theorem 4.8]{arnold2018finite}), and the following estimate holds
\begin{align}\label{wellposedness-withoutbc}
	\|\bm u\|+\|\grad\rot\bm u\|+\|\div\bm u\|_1+\|\curl\Delta\rot\bm u\|\leq C\|\bm f\|.
\end{align}
Since $\bm u\in D_{\mathscr L,\rot}$, $\Delta\rot\bm u=0$ on $\partial\Omega$. By the Poincar\'e inequality we have
\begin{align}\label{Deltarotu}
\|\Delta\rot\bm u\|\leq C\|\curl\Delta\rot\bm u\|\leq C\|\bm f\|.
\end{align}

Next we investigate the regularity of the solutions.
\begin{theorem}\label{regularity-withoutbc}
In addition to the assumptions on $\Omega$, we further assume that $\Omega$ is a  polygon. There exists a constant $\alpha>1/2$ such that the solution $\bm u$ of \eqref{prob-source} satisfies
	\[\bm u\in  H^{\alpha}(\Omega)\otimes \mathbb V\text{ and } \rot\bm u\in H^{1+\alpha}(\Omega),\] and it holds 
	\[\|\bm u\|_{{\alpha}}+\|\rot\bm u\|_{{1+\alpha}}\leq C\|\bm f\|.\]
	Moreover, if $\bm f\in H(\div;\Omega)$ and $\bm f\cdot\bm n\in H^{\alpha-1/2}(\partial\Omega)$, then $\div\bm u\in H^{1+\alpha}(\Omega)$ and it holds
	\begin{align*}
 \|\div\bm u\|_{1+\alpha}\leq C\left( \|\mathrm{div} \bm f\|+\|\bm f\cdot\bm n\|_{\alpha-1/2,\partial\Omega}\right).
\end{align*}
	\end{theorem}
\begin{proof}
It follows from the embedding
$H(\rot;\Omega)\cap H_0(\div;\Omega)\hookrightarrow H^{\alpha}(\Omega)\otimes\mathbb V$ with  $\alpha>1/2$ \cite{amrouche1998vector} that $\bm u\in H^{\alpha}(\Omega)\otimes\mathbb V$, and 
\[\|\bm u\|_{{\alpha}}\leq C\left( \|\bm u\|+\|\div \bm u\|+\|\rot\bm u\|\right).\]
Furthermore, by the Poincar\'e inequality we have
\[\|\bm u\|_{{\alpha}}\leq C\left( \|\bm u\|+\|\div \bm u\|+\|\grad\rot\bm u\|\right)\leq C\|\bm f\|.\]
Therefore it suffices to show that $\rot\bm u\in H^{1+\alpha}(\Omega)$. Since $\curl\Delta\rot\bm u\in L^2(\Omega)$, we have
\[-\Delta\rot\bm u\in L^2(\Omega).\]
Moreover, $\rot\bm u$ satisfies the boundary condition
$\rot \bm u=0.$
By the regularity of the Laplace problem \cite[Theorem 3.18]{Monk2003}, there exists an $\alpha>1/2$ such that
$\rot\bm u\in H^{1+\alpha}(\Omega)$, and 
\[\|\rot\bm u\|_{{1+\alpha}}\leq C \|\Delta\rot\bm u\|\leq C\|\bm f\|,\]
where we have used \eqref{Deltarotu}.

Multiplying both sides of \eqref{prob-source} by $-\grad q\in \grad H^1(\Omega)$ and integrating over $\Omega$, we obtain
\begin{align*}
(\grad\div \bm u,\grad q)=(\div \bm f,q)-\langle \bm f\cdot\bm n,q\rangle,
\end{align*} 
From the regularity of the Laplace problem \cite[Theorem 3.18]{Monk2003} again, there exists a constant $\alpha>1/2$ such that
\[\|\div\bm u\|_{1+\alpha}\leq C\left( \|\mathrm{div} \bm f\|+\|\bm f\cdot\bm n\|_{\alpha-1/2,\partial\Omega}\right).\]
\end{proof}

\subsection{Characterization of $H_{0}(\curl\div)$ and $H_{\div}(\curl\div)$}\label{characterization}
In the following, we characterize the spaces $H_{\div}(\curl\div;\Omega)$ and $H_{0}(\curl\div;\Omega)$. We denote by $C^{\infty}(\bar\Omega)$ the space of infinitely differentiable functions on $\bar\Omega$ and $C_0^{\infty}(\Omega)$ the space of infinitely differentiable functions with compact support on $\Omega$.

We start by defining a different norm.
\begin{lemma}
The following is an equivalent norm for $H(\grad\rot;\Omega)$:
\begin{align*}
\3bar\bm u\3bar_{H(\grad\rot;\Omega)}=\|\bm u\|+\|\rot\bm u\|+\|\grad\rot\bm u\|.
\end{align*}
\end{lemma}
\begin{proof}
	It is easy to check that $H(\grad\rot;\Omega)$ is a Banach space under the two norms $\3bar\cdot\3bar_{H(\grad\rot;\Omega)}$ and $\|\cdot\|_{H(\grad\rot;\Omega)}$. Applying the bounded inverse theorem, we obtain that the two norms are equivalent.
\end{proof}
\begin{theorem}\label{bounded1}
	Define $\gamma_{\tau,\rot}\bm u=\{\bm u\cdot \bm \tau,\rot\bm u\}.$ Then $\gamma_{\tau,\rot}$ is a linear bounded operator from $H(\grad\rot;\Omega)$ to $H^{-1/2}(\partial\Omega)\times H^{1/2}(\Omega)$ with the bound:
	\[\|\gamma_{\tau,\rot}\bm u\|_{H^{-1/2}(\partial\Omega) \times H^{1/2}(\partial\Omega)}\leq C \|\bm u\|_{H(\grad\rot;\Omega)}.\] 
\end{theorem}	
\begin{proof}
	Since $\gamma_\tau\bm u=\bm u\cdot \bm \tau$ is a linear bounded operator from $H(\rot;\Omega)$ to $H^{-1/2}(\partial\Omega)$ and $\text{tr} v=v|_{\partial\Omega} $ is a linear bounded operator from $H^1(\Omega)$ to $H^{1/2}(\partial\Omega)$, we have
	\begin{align*}
	&\|\gamma_{\tau,\rot}\bm u\|_{H^{-1/2}(\partial\Omega) \times H^{1/2}(\partial\Omega)}^2\\
		=&{\|\bm u\cdot \bm \tau\|_{H^{-1/2}(\partial\Omega)}^2+\|\rot\bm u\|_{H^{1/2}(\partial\Omega)}^2}\\
		\leq & C\|\bm u\|_{H(\rot;\Omega)}^2 +C\|\rot\bm u\|_{H^1(\Omega)}^2\\
		\leq & C\|\bm u\|_{H(\grad\rot;\Omega)}^2,
	\end{align*}
where we have used the equivalence between the norms $\|\cdot\|_{H(\grad\rot;\Omega)}$ and $\3bar\cdot\3bar_{H(\grad\rot;\Omega)}$.
\end{proof}
Similarly, we have
\begin{theorem}\label{bounded2}
	Define $\gamma_{n,\div}\bm u=\{\bm u\cdot \bm n,\div\bm u\}.$ Then $\gamma_{{n,\div}}$ is a linear bounded operator from $H(\curl\div;\Omega)$ to $H^{-1/2}(\partial\Omega)\times H^{1/2}(\Omega)$ with the bound:
	\[\|\gamma_{n,\div}\bm u\|_{H^{-1/2}(\partial\Omega) \times H^{1/2}(\partial\Omega)}\leq C \|\bm u\|_{H(\curl\div;\Omega)}.\] 
\end{theorem}	

\begin{theorem}\label{Surjection_gradrot}
The trace operator $\gamma_{\tau,\rot}$ is surjective from  $H(\grad\rot;\Omega)$ to $H^{-1/2}(\partial\Omega)\times H^{1/2}(\partial\Omega)$. That is, 
	for any $g_1\in H^{-1/2}(\partial\Omega)$ and $g_2\in H^{1/2}(\partial\Omega)$, there exists $\bm u\in H(\grad\rot;\Omega)$ such that $\bm u\cdot \bm \tau|_{\partial\Omega}=g_1$, $\rot\bm u|_{\partial\Omega}=g_2$, and 
	\[\|\bm u\|_{H(\grad\rot;\Omega)}\leq C\left(\|g_1\|_{H^{-1/2}(\partial\Omega)}+\|g_2\|_{H^{1/2}(\partial\Omega)}\right).\]
\end{theorem}
\begin{proof}
	For $g_2\in H^{1/2}(\partial\Omega)$, there exists $v\in H^1(\Omega)$ such that $v|_{\partial \Omega}=g_2$ and 
	\begin{align*}
		\|v\|_{H^1(\Omega)}\leq C\|g_2\|_{H^{1/2}(\partial \Omega)}.
	\end{align*}
	Take $x_0$ and $r$ such that $B(x_0,r)\subset \Omega$ and define
	\begin{align*}
		\eta_{r,x_0}(x)=\frac{1}{r^2}\eta\big(\frac{x-x_0}{r}\big),
	\end{align*}
	where 
		\begin{align*}
		\eta(x)=
		\begin{cases}
			C_1\exp(\frac{1}{|x|^2-1}), &|x|<1,\\
			0,&|x|\geq 1
		\end{cases}
	\end{align*}
	with $C_1=\left(\int_{\mathbb R^2}\exp\big(\frac{1}{|x|^2-1}\big)\d x\right)^{-1}$.
	Then we have $\eta_{r,x_0}(x)\in C_0^{\infty}(\Omega)$ and $\int_{\Omega}\eta_{r,x_0}(x)\d x =1$. 
	Let $C_0=\langle g_1,1\rangle_{\partial\Omega}-(v,1)$ and $\widetilde v=v+C_0\eta_{r,x_0}$. Then 
	\[(\widetilde v,1)=(v,1)+C_0(\eta_{r,x_0},1)=\langle g_1,1\rangle_{\partial\Omega}\]
	and 
	\begin{align}
		\|\widetilde v\|_{H^1(\Omega)}&\leq \|v\|_{H^1(\Omega)}+C \big(\langle g_1,1\rangle_{\partial\Omega}-(v,1)\big)\nonumber\\
		&\leq C\left(\|v\|_{H^1(\Omega)}+\|g_1\|_{H^{-1/2}(\partial \Omega)}\right)\nonumber\\
		&\leq C\left(\|g_2\|_{H^{1/2}(\partial \Omega)}+\|g_1\|_{H^{-1/2}(\partial \Omega)}\right).\label{tildev}
	\end{align}
Now we seek $w\in H^1(\Omega)$ such that 
\begin{align*}
	-\Delta w&=-\widetilde v\ \text{in } \Omega,\\
	\grad w\cdot \bm n&=g_1 \text{ on } \partial \Omega,
\end{align*}
where $\widetilde v$ and $g_1$ satisfy 
\[-(\widetilde v,1)+\langle g_1,1\rangle_{\partial \Omega}=0.\]
By virtual of the regularity result of the elliptic problem \cite[Theorem 3.18]{Monk2003},  we have
\begin{align}\label{w}
	\|w\|_{H^1(\Omega)}\leq C\left(\|\widetilde v\|+\|g_1\|_{H^{-1/2}(\partial\Omega)}\right).
\end{align}
Take $\bm u=\curl w$. Then $\bm u\in L^2(\Omega)$ and $\rot\bm u=\Delta w=\widetilde v\in H^1(\Omega)$, and hence $\bm u\in H(\grad\rot;\Omega).$
Restricted on $\partial \Omega$, 
\begin{align*}
	\bm u\cdot \bm \tau &=\curl w\cdot \bm \tau =\grad w\cdot \bm n=g_1,\\
	\rot \bm u&=\Delta w=\widetilde v=v=g_2.
\end{align*}
Combining \eqref{tildev} and \eqref{w}, we obtain
\begin{align*}
	\|\bm u\|_{H(\grad\rot;\Omega)}&\leq \|\bm u\|+\|\grad\rot \bm u\|\\
	&= \|\curl w\|+\|\grad \widetilde v\|\\
	&\leq \|w\|_{H^1(\Omega)}+\|\widetilde v\|_{H^1(\Omega)}\\
	&\leq C\left(\|\widetilde v\|+\|g_1\|_{H^{-1/2}(\partial\Omega)}\right)+\|\widetilde v\|_{H^1(\Omega)}\\
	&\leq C \left(\|g_1\|_{H^{-1/2}(\partial\Omega)}+\|g_2\|_{H^{1/2}(\partial\Omega)}\right).
	\end{align*}
\end{proof}
\begin{lemma}\label{dense}
	$ C^{\infty}(\overline{\Omega})\otimes \mathbb V$ is dense in $H(\grad\rot;\Omega)$.
\end{lemma}
\begin{proof}
We follow the proof of \cite[Theorem 2.4]{Monk2003}, which is based upon the following property of Banach spaces:
	
	\textit{A subspace $\mathcal M$ of a Banach space $M$ is dense in $M$ if and only if every element of $M^\prime$ that vanishes on $\mathcal M$ also vanishes on $M$.}
	
	Let $l\in [H(\grad\rot;\Omega)]^\prime$ and let $\bm l$ be the element of  $H(\grad\rot;\Omega)$ associated with $l$ by
	\[\langle l, \bm u\rangle=(\bm l,\bm u)+(\grad \rot\bm l,\grad\rot\bm u)\]
	Now, we assume that $l$ vanishes on $ C^{\infty}(\overline{\Omega})\otimes \mathbb V$. Denote $D\bm l=\grad\rot\bm l$. Let $\widetilde{\bm l}$ and $\widetilde{D \bm l}$ denote the extensions of $\bm l$ and $D\bm l$ by zero outside $\Omega$. The the above formula can be rewritten as 
		\[\int_{\mathbb R^2}\widetilde{\bm l}\cdot\bm \phi+\widetilde{D \bm l} \cdot \grad\rot\bm \phi\d x =0, \ \forall \bm \phi\in C^{\infty}(\mathbb R^2)\otimes \mathbb V,\]
which implies that, in the sense of distributions,
\begin{align}\label{I}
	\widetilde{\bm l}=\curl\div\widetilde{D \bm l}\in L^2(\mathbb R^2)\otimes \mathbb V.
\end{align}
Thus $\widetilde{D \bm l}\in H(\curl\div;\mathbb R^2)$, and then $\widetilde{D \bm l}\in H(\curl\div;\Omega)$. Now let $\mathcal O$ be a ball such that
$\Omega\subset\mathcal O$. Then $\mathcal O\backslash \overline\Omega$ is a bounded Lipschitz domain and $\widetilde{D \bm l}|_{\mathcal O\backslash \overline\Omega}=0$. Since $\gamma_n\bm u=\bm u\cdot\bm n$ is a bounded operator from $H(\div;\mathcal O\backslash \overline\Omega)\rightarrow H^{-1/2}(\partial \mathcal O\cup \partial \Omega)$, we have $\widetilde{D \bm l}\cdot\bm n|_{\partial\Omega}=0$. Then $\widetilde{D \bm l}\in H_0(\div;\Omega)$, and hence ${D \bm l}\in H_0(\div;\Omega)$. Similarly, we can get $\div{D \bm l}\in H_0^1(\Omega)$. Using ${\bm l}=\curl\div{D \bm l}$ and ${D \bm l}\in H_0(\div;\Omega)$, we have 
\begin{align*}
	\langle l,\bm u\rangle &=(\bm l,\bm u)+(\grad \rot\bm l,\grad\rot\bm u)\\
	&=(\curl\div{D \bm l},\bm u)-(\div D\bm l,\rot\bm u),\ \forall \bm u\in H(\grad\rot;\Omega)
	\end{align*}
Since $C_0^{\infty}(\Omega)$ is dense in $H_0^1(\Omega)$, there is a sequence $\{\phi_n\}_{n=1}^{\infty}\subset C_0^{\infty}(\Omega)$ such that $\phi_n\rightarrow \div D\bm l$ in $H_0^1(\Omega)$ as $n\rightarrow\infty$. Then
\begin{align*}
	\langle l,\bm u\rangle &=(\curl\div{D \bm l},\bm u)-(\div D\bm l,\rot\bm u)\\
&=\lim_{n\rightarrow \infty}\{(\curl \phi_n,\bm u)-( \phi_n,\rot\bm u)\}\\
	&=0, \ \forall \bm u\in H(\grad\rot;\Omega).
	\end{align*}
Therefore $l$ also vanishes on $H(\grad\rot;\Omega)$.
\end{proof}
Similarly, we can prove the following lemma.
\begin{lemma}\label{densecurldiv}
	$ C^{\infty}(\overline{\Omega})\otimes \mathbb V$ is dense in $H(\curl\div;\Omega)$.
\end{lemma}

 \begin{lemma}\label{integrationbyparts}
 For $\bm u\in H(\grad\rot;\Omega) \text{ and } \bm w\in H(\curl\div;\Omega)$, the following identity holds
\begin{align}\label{integretionbyparts-identity}
 (\bm u,\curl\div \bm w)+(\grad\rot\bm u,\bm w)=\langle \bm u\cdot\bm \tau,\div \bm w\rangle_{\partial\Omega}+\langle \bm w\cdot\bm n,\rot\bm u\rangle_{\partial\Omega}.
\end{align}
 \end{lemma}
\begin{proof}
	It is easy to check that \eqref{integretionbyparts-identity} holds for smooth functions $\bm u,\bm w$. 
	By Lemma \ref{dense}, Lemma \ref{densecurldiv}, Theorem \ref{bounded1}, and Theorem \ref{bounded2}, we can prove \eqref{integretionbyparts-identity} for $\bm u\in H(\grad\rot;\Omega)$ and $\bm w\in H(\curl\div;\Omega)$.
\end{proof}

\begin{lemma}\label{H0(curldiv)}
	 The space $H_{0}(\curl\div;\Omega)$ can be characterized as 
	\[H_{0}(\curl\div;\Omega)=\{\bm w\in H(\curl\div;\Omega):\div \bm w=0,\ \bm u\cdot \bm n=0 \text{ on } \partial \Omega\}.\]
\end{lemma}
\begin{proof}
	If $\bm w\in H_{0}(\curl\div;\Omega)$, i.e., the domain of the adjoint of $(\grad\rot,H(\grad\rot;\Omega))$, then there exists $\bm v\in L^2(\Omega)\otimes \mathbb V$ such that
	\[(\grad\rot\bm u,\bm w)=-(\bm u,\bm v),\quad \bm u\in H(\grad\rot;\Omega).\]
	Such a function $\bm w$ must be in $H(\curl\div;\Omega)$ and satisfies $\curl\div\bm w=\bm v.$
	Therefore, $\bm w$ belongs to $H_{0}(\curl\div;\Omega)$ if and only if
	\[(\grad\rot\bm u,\bm w)=-(\bm u,\curl\div\bm w),\quad \bm u\in H(\grad\rot;\Omega).\]
	From Lemma \ref{integrationbyparts}, the above identity holds if and only if 
	\[\langle \bm u\cdot\bm \tau,\div \bm w\rangle_{\partial\Omega}+\langle \bm w\cdot\bm n,\rot\bm u\rangle_{\partial\Omega}=0.\]
	which holds when $\gamma_{n,\div}\bm w=\{\bm w\cdot \bm n, \div \bm w\}=0$ since $\gamma_{\tau,\rot}$ is surjective from  $H(\grad\rot;\Omega)$ to $H^{-1/2}(\partial\Omega)\times H^{1/2}(\partial\Omega)$ (see Theorem \ref{Surjection_gradrot}).	\end{proof}

\begin{lemma}
	 The space $H_{\div}(\curl\div;\Omega)$ can be characterized as 
	\[H_{\div}(\curl\div;\Omega)=\{\bm w\in H(\curl\div;\Omega):\div \bm w=0\text{ on } \partial \Omega\}.\]
\end{lemma}
\begin{proof}
	The proof is similar to that of Lemma \ref{H0(curldiv)}.
	\end{proof}

\section{Convergence analysis and explanations of spurious solutions}

 In this section, we prove that the mixed formulations provide correctly convergent solutions for both the source and the eigenvalue problems. Therefore the different solutions by other schemes in Section 2 are spurious. We also provide an explanation of the spurious solutions.  Here we assume that $\Omega$ is simply-connected, and hence $\mathfrak H_{\rot}$ vanishes. Define
\[\mathfrak Z_h=\{\bm v_h\in V_h: \grad\rot\bm v_h=0\}.\]
From the finite element complexes in \cite{hu2020simple} and the vanishing $\mathfrak H_{\rot}$, we have $\mathfrak Z_h=\grad S_h$.

\subsection{Source problem}

We first show the convergence for the source problems.

According to \cite[Theorem 5.4]{arnold2018finite},  Scheme \ref{algo1} is stable if the following discrete Poincar\'e inequality holds. The discrete Poincar\'e inequality for $V_{h}$ is due to special structures of the $H(\grad\curl)$-conforming elements and the complexes in \cite{hu2020simple,WZZelement}, i.e., 1) these elements are subspaces of the N\'ed\'elec elements, 2) the 0-forms in the complexes \cite{hu2020simple} are the Lagrange elements (standard finite element differential forms).
\begin{lemma}[discrete Poincar\'e inequality for $V_{h}$]
	For $\bm v_h\in V_h\cap \mathfrak Z_h^{\perp},$ we have
	\begin{equation}\label{poincare}
	\|\bm v_h\|_{H(\grad\rot;\Omega)}\leq C\|\grad\rot\bm v_h\|,
	\end{equation}
	where $C$ is a constant independent of $h$.
\end{lemma}
\begin{proof}
Let $P_{k}^-\Lambda^{l}$ be the standard finite element differential $k$-forms on triangles \cite{arnold2018finite}, i.e., $l=0$ corresponds to the Lagrange element and $l=1$ corresponds to the N\'ed\'elec elements. Due to the interelement continuity, $V_{h}\subset P_{k}^-\Lambda^{1}$. Moreover, $\mathfrak Z_h=\grad S_h=\grad P_{k}^-\Lambda^{0}$. Then \eqref{poincare} follows from the discrete Poincar\'e inequality of $P_{k}\Lambda^{1}$ and $P_{k}\Lambda^{0}$. 
\end{proof}

\begin{theorem}
Suppose that $\Omega$ is a simply-connected Lipschitz polygon. 
	Let $(\bm u_h, \sigma_h)$ be the numerical solution of Scheme \ref{algo1} and $(\bm u,\sigma)$ be the exact solution of the problem \eqref{prob-source}. Then
	\[\|\bm u-\bm u_h\|_{H(\grad\rot;\Omega)}+\|\sigma-\sigma_h\|_1\leq C h^{\alpha}(\|\bm u\|_{\alpha}+\|\rot\bm u\|_{1+\alpha}+\|\sigma\|_{1+\alpha}).\]
\end{theorem}
\begin{proof}
When the harmonic function space vanishes, it follows from the proof of \cite[Theorem 5.5]{arnold2018finite} and the discrete Poincar\'e inequality that 
\[\|\bm u-\bm u_h\|_{H(\grad\rot;\Omega)}+\|\sigma-\sigma_h\|_1\leq C(\inf_{\tau_h\in S_h}\|\sigma-\tau_h\|+C\inf_{\bm v_h\in V_h}\|\bm u-\bm v_h\|).\] 
Let $\pi_h$ and $\Pi_h$ be the canonical interpolations to $S_h$ and $V_h$. From their approximation properties \cite{WZZelement,hu2020simple}, we have
\begin{align*}
\|\bm u-\bm u_h\|_{H(\grad\rot;\Omega)}+\|\sigma-\sigma_h\|_1\leq C h^{\alpha}(\|\bm u\|_{\alpha}+\|\rot\bm u\|_{1+\alpha}+\|\sigma\|_{1+\alpha}).
\end{align*}
\end{proof}

Now we are in a position to explain the spurious solutions.

Let $H_{n,\rot}^1(\rot;\Omega)=H^1(\rot;\Omega)\cap X_{\rot}$
denote the space of $H^1(\rot;\Omega)$ vector fields with vanishing normal components and $\rot$ on the boundary, which is a closed subspaces of $ H^1(\rot;\Omega)$. Clearly, $H_{n,\rot}^1(\rot;\Omega)\subset X_{\rot}$. Moreover, by the Poincar\'e inequality and the fact that $\|\grad \bm u\|^2=\|\rot\bm u\|^2+\|\div\bm u\|^2$  for $\bm u\in H_{n,\rot}^1(\rot;\Omega)$ \cite{arnold2018finite}, we have, for $\bm u\in H_{n,\rot}^1(\rot;\Omega)$,
\begin{align*}
	\|\grad\rot \bm u\|^2+\|\grad \bm u\|^2\leq C(\|\grad\rot \bm u\|^2+\|\div \bm u\|^2)\leq C(\|\grad\rot \bm u\|^2+\|\grad \bm u\|^2).
	\end{align*}
Therefore, the restriction of the $X$-norm to $H_{n,\rot}^1(\rot;\Omega)$ is equivalent to the full norm of $H^1(\rot;\Omega)$. It follows from the fact $H_{n,\rot}^1(\rot;\Omega)$ is closed in $H^1(\rot;\Omega)$ that $H_{n,\rot}^1(\rot;\Omega)$ is a closed subspace of $X_{\rot}$. To prove $X_{\rot}\neq H_{n,\rot}^1(\rot;\Omega)$, it suffices to find a function in $X_{\rot}$ which is not in $H_{n,\rot}^1(\rot;\Omega)$. Consider the function $\phi\in H^1(\Omega)$ such that $\Delta \phi\in L^2(\Omega)$ and $\frac{\partial \phi}{\partial n}=0$ on $\partial \Omega$.  When $\Omega$ is a nonconvex polygonal domain, we have $\phi \notin  H^2(\Omega)$. Setting $\bm u=\grad \phi$, we see that $\bm u\in X_{\rot}$ but $\bm u\notin  H_{n,\rot}^1(\rot;\Omega)$. For any such function $\bm u$, we have $\inf_{\bm v\in H_{n,\rot}^1(\rot;\Omega)}\|\bm u-\bm v\|_X=\delta_{\bm u}>0$, where $\delta_{\bm u}$ is the distance of $\bm u$ from $H_{n,\rot}^1(\rot;\Omega)$. Therefore, if the finite element space $V_h$ is contained in $H_{n,\rot}^1(\rot;\Omega)$, then the numerical solution $\bm u_h\in V_h$ can not converge to $\bm u$ in general.
The mixed variational formulation, however, does not suffer from this restriction, and hence does not lead to spurious solutions. 
On the other hand, the choice of finite elements are crucial for the success of the mixed formulations. The finite elements shown above are stable as they fit into complexes.

Since $\dim \mathfrak H_{\rot}=\beta_1=1$ on $\Omega_2$, there is a zero eigenvalue on $\Omega_2$ corresponding to the harmonic forms in $\mathfrak H_{\rot}$. However, Scheme \ref{algo7} and Scheme \ref{algo8} fail to capture this zero eigenmode. The same issue occurs for the numerical solutions of the problem \eqref{prob2}. Generally speaking, this is because the formulations and finite elements do not respect the underlying cohomological structures of the $\grad\rot$ complex.

\subsection{Eigenvalue problem}
To obtain the convergence estimate for Scheme \ref{algo5}, we rewrite \eqref{mvprob-eig} and Scheme \ref{algo5} as follows.

Seek $(\widetilde\lambda,\bm u,\sigma)\in \mathbb R\times H_{\rot}(\grad\rot;\Omega)\times H^1(\Omega)$ such that
\begin{equation}\label{mvprob-eig-u}
\begin{split}
(\grad\rot\bm u,\grad\rot\bm v)+(\bm u,\bm v)+(\grad\sigma,\bm v)&= \widetilde\lambda(\bm u,\bm v), \ \forall \bm v\in  H_{\rot}(\grad\rot;\Omega),\\
(\bm u,\grad \tau)-(\sigma,\tau)&=0,\ \forall \tau\in H^1(\Omega).
\end{split}
\end{equation}

Find $(\widetilde\lambda_h,\bm u_h,\sigma_{h}) \in \mathbb R\times V_h\times S_{h}$,  such that
\begin{equation}\label{eig-s1-u}
\begin{split}
(\grad\rot\bm u_h,\grad\rot\bm v_h)+(\bm u_h,\bm v_h)+(\grad\sigma_h,\bm v_h)&= \widetilde\lambda_h(\bm u_h,\bm v_h), \ \forall \bm v_h\in V_h,\\
(\bm u_h,\grad \tau_h)-(\sigma_h,\tau_h)&=0,\ \forall \tau_h\in S_h.
\end{split}
\end{equation}

Note that $\widetilde\lambda=\lambda+1$ and $\widetilde\lambda_h=\lambda_h+1$. 
We also consider the corresponding source problem and its finite element discretization.

Seek $(\bm u,\sigma)\in H_{\rot}(\grad\rot;\Omega)\times H^1(\Omega)$ such that
\begin{equation}\label{mvprob-f-u}
\begin{split}
(\grad\rot\bm u,\grad\rot\bm v)+(\bm u,\bm v)+(\grad\sigma,\bm v)&=(\bm f,\bm v), \ \forall \bm v\in  H_{\rot}(\grad\rot;\Omega),\\
(\bm u,\grad \tau)-(\sigma,\tau)&=0,\ \forall \tau\in H^1(\Omega).
\end{split}
\end{equation}

Find $(\bm u_h,\sigma_{h}) \in  V_h\times S_{h}$,  such that
\begin{equation}\label{source-s1-u}
\begin{split}
(\grad\rot\bm u_h,\grad\rot\bm v_h)+(\bm u_h,\bm v_h)+(\grad\sigma_h,\bm v_h)&= (\bm f,\bm v_h), \ \forall \bm v_h\in V_h,\\
(\bm u_h,\grad \tau_h)-(\sigma_h,\tau_h)&=0,\ \forall \tau_h\in S_h.
\end{split}
\end{equation}

By suitable modification to the proof of \cite[Theorem 4.8]{arnold2018finite}) and \cite[Theorem 4.9]{arnold2018finite}), we can show the problem \eqref{mvprob-f-u} is well-posed, and \eqref{wellposedness-withoutbc} holds.
Similar to the problem \eqref{prob-source}, we can get the same regularity estimate as Theorem \ref{regularity-withoutbc}.

Define the solution operators $T:L^2(\Omega)\otimes\mathbb V\rightarrow L^2(\Omega)\otimes\mathbb V$ and $S:L^2(\Omega)\otimes\mathbb V\rightarrow H^1(\Omega)$ by
\[T\bm f:=\bm u\text{ and }S\bm f:=\sigma.\]
We also define the discrete solution operators $T_h:L^2(\Omega)\otimes\mathbb V\rightarrow L^2(\Omega)\otimes\mathbb V$ and $S_h:L^2(\Omega)\otimes\mathbb V\rightarrow H^1(\Omega)$ by
\[T_h\bm f:=\bm u_h\text{ and }S_h\bm f:=\sigma_h.\]

From \eqref{mvprob-f-u} and \eqref{source-s1-u}, these operators are bounded and satisfy
\begin{align}\label{boundedness}
\begin{aligned}
&\|T\bm f\|_{H(\grad\rot;\Omega)}\leq \|\bm f\|, \	\|T_h\bm f\|_{H(\grad\rot;\Omega)}\leq \|\bm f\|, \\
& \|S\bm f\|_1\leq \|\bm f\|, \text{ and }\|S_h\bm f\|_1\leq \|\bm f\|.
\end{aligned}
\end{align}

We  have the following orthogonality:
\begin{align}
\begin{aligned}
	(T\bm f- T_h\bm f, \bm v_h)_{H(\grad\rot;\Omega)}+(\grad(S\bm f- S_h\bm f),\bm v_h)&=0, \ \forall \ \bm v_h\in V_h,\\
	(T\bm f- T_h\bm f, \grad \tau_h)-(S\bm f- S_h\bm f,\tau_h)&=0,\ \forall \ \tau_h\in S_h.\label{ortho-property}	
\end{aligned}
\end{align}
Taking $\bm v_h=\grad \tau_h$ in the first equation of \eqref{ortho-property} and subtracting the second equation from the first equation, we obtain
\begin{equation}\label{ortho-property-3}
	(\grad(S\bm f- S_h\bm f),\grad\tau_h)+(S\bm f- S_h\bm f,\tau_h)=0.
\end{equation}

If we can prove $\|T-T_h\|_{L^2\otimes\mathbb V\rightarrow L^2\otimes \mathbb V}\rightarrow 0$, then from the spectral approximation theory in \cite{babuvska1991eigenvalue}, the eigenvalues of \eqref{eig-s1-u} converge to the eigenvalues of \eqref{mvprob-eig-u}, and hence, the eigenvalues of Scheme \ref{algo4} converge to the eigenvalues of \eqref{mvprob-eig}.

To obtain the uniform convergence, we present some preliminary results. First, we will need the Hodge mapping $\check{\bm q}$ of $\bm q_{h}$, which is a continuous approximation of a finite element function \cite{he2019generalized,hiptmair2002finite}. In fact, for any ${\bm q}_h\in \mathfrak Z_h^\perp\cap V_h$, there exists $\check{\bm q}\in H(\rot;\Omega)\cap H_{0}(\div0; \Omega)$ satisfying $\rot \bm q_h=\rot\check{\bm q}$. Here $H_{0}(\div0; \Omega):=\{\bm{u}\in H(\div; \Omega): \div \bm{u}=0\}$.

\begin{lemma}\label{checkqconv}   There exists a constant $C$ independent of $\bm q_h$ and $h$ such that 
	\[\|\check{\bm q}-{\bm q}_h\|\leq Ch^{\alpha}\|\rot\check{\bm q}\|, \quad \forall {\bm q}_h\in \mathfrak Z_h^\perp\cap V_h.\]
\end{lemma}
\begin{proof}
Proceeding as the proof of \cite[Lemma 7.6]{Monk2003} or \cite[Lemma 4.5]{hiptmair2002finite}, we can complete the proof.
\end{proof}

\begin{lemma}\label{sigmaest}
	Under the assumptions of Theorem \ref{regularity-withoutbc}, the solutions of  \eqref{mvprob-f-u} and \eqref{source-s1-u} satisfy 
\[\left\|\sigma-\sigma_{h}\right\|\leq Ch^{\alpha}\|\bm f\|.\]
\end{lemma}
\begin{proof}
We introduce the following auxiliary problem: find $\check{\sigma} \in H^{1}(\Omega)$ such that
\begin{align}\label{aux-sigma} 
	(\grad \check{\sigma}, \grad \tau)+(\check{\sigma}, \tau)=\left(\sigma-\sigma_{h}, \tau\right), \quad \forall\tau \in H^{1}(\Omega).
\end{align}
The discrete problem is to find $\check{\sigma}_{h} \in S_{h}$ such that
$$
(\grad \check{\sigma}_h, \grad \tau_h)+(\check{\sigma}_h, \tau_h)=\left(\sigma-\sigma_{h}, \tau_h\right), \quad \forall\tau_h \in S_h.
$$
According to Theorem 3.18 in \cite{Monk2003} again, there exists the same constant $\alpha>1/2$ such that
$$
\|\check{\sigma}\|_{1+\alpha} \leq C\left\|\sigma-\sigma_{h}\right\|.
$$
From the Ce\'a lemma, we have
\[\left\|\check\sigma-\check\sigma_{h}\right\|_{1}\leq C h^{\alpha}\|\check{\sigma}\|_{1+\alpha}\leq C h^{\alpha}\left\|\sigma-\sigma_{h}\right\|.\]
Taking $\tau=\sigma-\sigma_h$ in \eqref{aux-sigma} and applying \eqref{ortho-property-3}, we obtain
\begin{align*}
	&\left(\sigma-\sigma_{h}, \sigma-\sigma_{h}\right) =\left(\check{\sigma}, \sigma-\sigma_{h}\right)+\left(\grad \check{\sigma}, \grad\left(\sigma-\sigma_{h}\right)\right) \\ 
	&=\left(\check{\sigma}-\check{\sigma}_{h}, \sigma-\sigma_{h}\right)+\left(\grad\left(\check{\sigma}-\check{\sigma}_{h}\right), \grad\left(\sigma-\sigma_{h}\right)\right) \\ 
	& \leq \left\|\check\sigma-\check\sigma_{h}\right\|_{1} \left\|\sigma-\sigma_{h}\right\|_{1}\leq C h^{\alpha}\left\|\sigma-\sigma_{h}\right\|\left\|\sigma-\sigma_{h}\right\|_{1},
\end{align*}
which together with \eqref{boundedness} leads to
\[\left\|\sigma-\sigma_{h}\right\|\leq Ch^{\alpha}\left\|\sigma-\sigma_{h}\right\|_{1}\leq Ch^{\alpha}\|\bm f\|.\]
\end{proof}
We are now in a position to estimate $\|T-T_h\|$.
\begin{theorem}
	Under the assumptions of Theorem \ref{regularity-withoutbc}, we have
\[\left\|T-T_{h}\right\|_{L^2\otimes\mathbb V\rightarrow L^2\otimes \mathbb V}\leq Ch^{\alpha}.\]
\end{theorem}
\begin{proof}
We shall  prove $\|\bm u-\bm u_h\|\leq Ch^{\alpha}\|\bm f\|.$ 
Denote $\tilde{\bm u}=T(\bm u-\bm u_h)$, $\tilde{\bm u}_h=T_h(\bm u-\bm u_h)$, $\tilde\sigma=S(\bm u-\bm u_h)$, and $\tilde\sigma_h=S_h(\bm u-\bm u_h)$. 
Proceeding as the dual argument, we have
\begin{align*}
&\left(\bm{u}-\bm{u}_{h}, \bm{u}-\bm{u}_{h}\right) =\left(\tilde{\bm{u}}, \bm{u}-\bm{u}_{h}\right)_{H(\grad\rot;\Omega)}+\left(\grad \tilde{\sigma}, \bm{u}-\bm{u}_{h}\right) \\ 
=&\left(\tilde{\bm{u}}-\tilde{\bm{u}}_{h}, \bm{u}-\bm{u}_{h}\right)_{H(\grad\rot;\Omega)}+\left(\grad\tilde{\sigma}, \bm{u}-\bm{u}_{h}\right)-\left(\grad (\sigma- \sigma_{h}), \tilde{\bm{u}}_{h}\right) \text{\quad\quad\ \ \ by \eqref{ortho-property} }\\
=&\left(\tilde{\bm{u}}-\tilde{\bm{u}}_{h}, \bm{u}-\Pi_{h} \bm{u}\right)_{H(\grad\rot;\Omega)}+\left(\grad\left(\tilde{\sigma}-\tilde{\sigma}_{h}\right), \bm{u}-\Pi_{h} \bm{u}\right)+\left(\sigma-\sigma_{h}, \tilde{\sigma}_{h}\right)\text{\quad by \eqref{ortho-property},\ \eqref{ortho-property-3} }\\
&-\left(\grad (\sigma-\sigma_{h}), \tilde{\bm{u}}_{h}\right)=:\text{I}+\text{II}+\text{III}+\text{IV},
\end{align*}
where $\Pi_{h}$ is the canonical interpolation to $V_h$.
Applying \eqref{boundedness} and the approximation property of $\Pi_h$, we have
\begin{align*}
\text{I}+\text{II}&\leq \|\bm{u}-\Pi_{h} \bm{u}\|_{H(\grad\rot;\Omega)}\left(\|\tilde{\bm{u}}-\tilde{\bm{u}}_{h}\|_{H(\grad\rot;\Omega)}+\|\grad\left(\tilde{\sigma}-\tilde{\sigma}_{h}\right)\|\right)\\
&\leq C\|\bm{u}-\Pi_{h} \bm{u}\|_{H(\grad\rot;\Omega)}\|\bm u-\bm u_h\|\\
&\leq Ch^{\alpha}\|\bm u-\bm u_h\|\left(\|\bm u\|_{\alpha}+\|\rot\bm u\|_{1+\alpha}\right)\\
&\leq Ch^{\alpha}\|\bm f\|\|\bm u-\bm u_h\|.
\end{align*}
Using Lemma \ref{sigmaest}, we get
\[\text{III}\leq Ch^{\alpha}\|\bm f\|\|\tilde \sigma_h\|\leq  Ch^{\alpha}\|\bm f\|\|\bm u-\bm u_h\|.\]
Now it remains to estimate IV. We decompose $\tilde{\bm u}_h= \bm p_h+\bm q_h$ with $\bm p_h\in \mathfrak Z_h$ and $\bm q_h\in \mathfrak Z_h^{\perp}\cap V_h$. Then 
\[\text{IV}=\left(\grad(\sigma-\sigma_h),\tilde{\bm u}_h\right)=\left(\grad(\sigma-\sigma_h),{\bm q}_h\right)+\left(\grad(\sigma-\sigma_h),{\bm p}_h\right)=:\text{IV}_{\text{I}}+\text{IV}_{\text{II}}.\]
 Applying  Lemma \ref{checkqconv}, we obtain
\[\text{IV}_{\text{I}}=\left(\grad(\sigma-\sigma_h),\bm q_h-\check{\bm q}\right)\leq Ch^{\alpha}\|\rot \bm q_h\|\|\bm f\|\leq Ch^{\alpha}\|\rot\tilde{\bm u}_h\|\|\bm f\|\leq Ch^{\alpha}\|\bm u-\bm u_h\|\|\bm f\|.\]
Since $\mathfrak Z_h=\grad S_h$, there exists a function $\phi_h\in S_h$ satisfying $(\phi_h,1)=0$ and $\bm p_h=\grad\phi_h$. By \eqref{ortho-property-3}, we get
 \begin{align*}
	\text{IV}_{\text{II}}=&\big(\grad(\sigma-\sigma_h),\grad\phi_h \big)
	=-\big(\sigma-\sigma_h,\phi_h\big)\leq C\|\sigma-\sigma_h\|\|\grad\phi_h\|\\
	=&C\|\sigma-\sigma_h\|\|\bm p_h\|\leq C\|\sigma-\sigma_h\|\|\tilde{\bm u}_h\|\leq Ch^{\alpha}\|\bm f\|\|\bm u-\bm u_h\|.
\end{align*}
Collecting all the estimates, we complete the proof.
\end{proof}

\section{Conclusion}
In this paper, we showed that inappropriate discretizations for high order curl problems lead to spurious solutions. Comparing to the spurious solutions of the Maxwell and biharmonic equations, the differential structure and the high order nature of these operators bring in new ingredients, e.g., the failure of capturing zero eigenmodes by some finite element methods. This issue is rooted in the cohomological structures of the $\grad\rot$ complex, which further shows that the algebraic structures in the (generalized) BGG complexes \cite{arnold2020complexes} have a practical impact for computation (even though the $\grad\rot$ complex is the simplest example in the framework). Reliable numerical methods, e.g., mixed formulations with elements in a complex, should respect these cohomological structures.  The algebraic and analytic results for the $\grad\rot$ complex also have various consequences for the high order curl problems, e.g., equivalent variational forms and the well-posedness etc.

Future directions include, e.g., bounded cochain projections for the finite elements and kernel-capturing algebraic solvers.

\bibliographystyle{siam}
\bibliography{reference}
\vspace{1cm}

  \end{document}